\documentclass[reqno]{amsart}
\usepackage{color}
\usepackage{amssymb}
\usepackage{enumerate}
\newtheorem{theorem}{Theorem}[section]
\newtheorem{lemma}[theorem]{Lemma}
\newtheorem{corollary}[theorem]{Corollary}
\newtheorem{proposition}[theorem]{Proposition}

\newtheorem{definition}[theorem]{Definition}

\newtheorem{Concluding Remark}[theorem]{Concluding Remark}

\newcommand{\no}{\nonumber}

\numberwithin{equation}{section}
\usepackage{mathrsfs}
\usepackage{xcolor}  

\begin{document}
	
	\title[ On the modularity of the odd rank generating functions]
	{ On the modularity of the odd rank generating functions}
	\maketitle
	\begin{center}
		
		Renrong Mao  
		
		Department of Mathematics,\\
		Soochow University, \\
		Suzhou, 215006,
		People's Republic of China\\[6pt]

		Email: rrmao@suda.edu.cn

	\end{center}
	\textbf{Abstract:}
 To provide partition-theoretic interpretations to Watson's the third-order mock theta function $\omega(q)$, Andrews defined the odd Durfee symbols and odd ranks.
 Motivated by Andrews' work, arithmetic properties of odd ranks are widely studied recently. In this paper, we obtain transformation formulas of the odd rank generating functions,
which are used to construct families of weak Maass forms and weakly modular forms. As an application,
we provide explicit identities for odd ranks modulo 5, analogous to Ramanujan’s classical partition identities.	
	
\vspace{0.2cm}	
\noindent\textbf{Mathematics Subject Classification (2020):} 11P83, 11F37, 11F11, 05A19

\vspace{0.2cm}	
\noindent	\textbf{Keywords:} odd Durfee symbol,
		 odd ranks, weak Maass form, Appell-Lerch sums
	
%
	
%

	\section{Introduction}
	\allowdisplaybreaks
		A partition of a positive integer $n$ is a sequence of non-increasing positive integers whose sum equals $n$. Let
	$p(n)$ denote the number of partitions of $n$.
	Ramanujan \cite{ra1} proved the following famous congruences:
	\begin{align}
		p(5n+4)&\equiv 0\pmod 5\label {1i1}\\
		p(7n+5)&\equiv 0\pmod 7\label {1i2}\\
		p(11n+6)&\equiv 0\pmod{11}\label {1i3}.
	\end{align}
	Moreover, Ramanujan also found two identities for $p(5n+4)$ and $p(7n+5)$:
	\begin{align}
		&\sum_{n=0}^{\infty}p(5n+4)q^n=5\frac{(q^5;q^5)^5_\infty}{(q;q)^6_\infty},\label{1id5}
		\intertext{and}
		&\sum_{n=0}^{\infty}p(7n+5)q^n=7\frac{(q^7;q^7)^3_\infty}
		{(q;q)^4_\infty}+49q\frac{(q^7;q^7)^7_\infty}{(q;q)^8_\infty}.\label{1id7}
	\end{align}
Here and for the rest of the paper, we use the notations
	\begin{align*}
		(a)_\infty:=(a;q)_\infty:&=\prod_{n=0}^\infty
		(1-aq^n),\\[6pt]
		(a_1,a_2,\ldots,a_k)_\infty :=(a_1,a_2,\ldots,a_k;q)_\infty
		:&=(a_1;q)_\infty
		(a_2;q)_\infty \cdots (a_k;q)_\infty,\\[6pt]
		[a_1,a_2,\ldots,a_k]_\infty:=[a_1,a_2,\ldots,a_k;q]_\infty
		:&=(a_1,q/a_1,a_2,q/a_2,\ldots,
		a_k, q/a_k;q)_\infty,\\[6pt]
		J_m
		:&=(q^m;q^m)_\infty,\\
		J_{a,b}
		:&=(q^a,q^{b-a},q^b;q^b)_\infty
	\end{align*}
	with $q:=e^{2\pi i \tau}$ and $\textrm{Im}(\tau)>0$. G. H. Hardy referred to \eqref{1id5} as ``Ramanujan's most beautiful identity" (cf. \cite[p. xvvv]{collect}).
	
	In order to explain Ramanujan's congruences combinatorially, Dyson \cite{dyson} defined the rank of a partition to be the largest part minus the number of parts.
	Let $N(m,n)$  denote the number of partitions of $n$ with rank $m$ and
	\begin{align}
		N(m,k,n):=\sum_{t\equiv m\pmod{k}}	N(t,n).\no
	\end{align}
	 Dyson conjectured that
	\begin{align}
		N(m,5,5n+4)&=\frac{p(5n+4)}{5}\label{1i4}
		\intertext{for all $0\leq m\leq 4$ and for $0\leq j\leq 6$}
		N(j,7,7n+5)&=\frac{p(7n+5)}{7}\label{1i5} .
	\end{align}
Obviously, equation \eqref{1i4} (resp. \eqref{1i5}) implies \eqref{1i1} (resp. \eqref{1i2}). 
Using $q$-series techniques, Atkin and Swinnerton-Dyer first found a proof of
Dyson's conjecture in \cite{as}.
	They also obtained the generating function for every rank difference $N(b,l,ln+d)-N(c,l,ln+d)$ with $l=5$ or $7$ and $0\leq b,c,d<l$.
	As noted by Garvan \cite{gar}, Dyson's conjecture also follows from an identity in Ramanujan's Lost Notebook \cite[p. 20]{ra2} (see also \cite[Eq. (2.1.17)]{lost}):
	\begin{align}
	R(\zeta_5, q) &=\frac{[q^{10}; q^{25}]_\infty J_{25}}{[q^5; q^{25}]_\infty^2} + (\zeta_5 + \zeta_5^{-1} - 2) \left(-1 + \sum_{n=0}^\infty \frac{q^{25n^2}}{(q^5; q^{25})_{n+1}(q^{20}; q^{25})_n}
	\right)\no\\&\quad + q  \frac{J_{25}}{[q^5; q^{25}]_\infty}+ (\zeta_5 + \zeta_5^{-1}) q^2   \frac{J_{25}}{[q^{10}; q^{25}]_\infty} 
	- (\zeta_5 + \zeta_5^{-1}) q^3 \no\\&\quad\times\left\{ \frac{[q^{5}; q^{25}]_\infty J_{25}}{[q^{10}; q^{25}]_\infty^2} - \frac{\zeta_5^2 + \zeta_5^{-2} - 2}{q^5} \left( \sum_{n=0}^\infty \frac{q^{25n^2}}{(q^{10}; q^{25})_{n+1}(q^{15}; q^{25})_n}-1\right) \right\},\label{idlo}
	\end{align}
	where $\zeta_c := e^{2\pi i/c}$ and \begin{align*}
		R(z ; q):&=\sum_{n=0}^{\infty} \sum_{m \in \mathbb{Z}} N(m, n) z^m q^n.
	\end{align*}
	
	
	Motivated by these works, ranks of partitions are widely studied. In \cite{ann}, Bringmann and Ono studied the modular properties of the generating function
	of Dyson's rank. As an application, they used these generating functions to construct families of weak Maass forms.
	To give their result, we need the definition of weak Maass forms of half-integral weight $k\in\frac{1}{2}\mathbb{Z}$.
	Throughout, we
	let $\tau=x+iy\in\mathbb{C}$ with $y>0,$ the hyperbolic Laplacian of weight $k$ is given by
	$$\Delta_k:=-y^2\left(\frac{\partial^2}{\partial x^2}+\frac{\partial^2}{\partial y^2}\right)+iky
	\left(\frac{\partial}{\partial x}+i\frac{\partial}{\partial y}\right).$$
	For odd integer $d$, define
	\begin{align*}
		\epsilon_d:=
		\left\{
		\begin{aligned}
			&1, \quad\textrm{if $d\equiv1\pmod{4}$},\\
			&i, \quad\textrm{if $d\equiv3\pmod{4}$}.
		\end{aligned}
		\right.
	\end{align*}
	\begin{definition}
		A weak Maass form of weight $k$ and Nebentypus $\chi$ on a subgroup $\Gamma\subset\Gamma_{0}(4)$ is
	a smooth function $f:\mathbb{H}\rightarrow\mathbb{C}$ satisfying:
	\begin{enumerate}[(1)]
		\item For all $A=\begin{pmatrix}
			\alpha & \beta  \\
			\gamma & \delta  \\
		\end{pmatrix}\in\Gamma,$ $f(A\tau)=\left(\frac{\gamma}{\delta}\right)^{2k}\epsilon_\delta^{-2k}\chi(\delta)(\gamma\tau+\delta)^kf(\tau).$
		\item
		We have that $\Delta_kf=0.$
		\item
		The function $f(\tau)$ has at most linear exponential growth at all cusps of $\Gamma.$
	\end{enumerate}
	We call a function $f:\mathbb{H}\rightarrow\mathbb{C}$ a weakly holomorphic modular form if
	the condition $(2)$ is replaced by that $f$ is holomorphic on $\mathbb{H}$. 	
	\end{definition}
Usually, a weak Maass form $f$ has a Fourier expansion of the form
	\begin{align}\label{inf}
		f(\tau)=\sum_{n=n_0}^\infty a(n)q^n+\sum_{n=1}^\infty b(n)\Gamma(1-k;4\pi ny)q^{-n}
	\end{align}
	where $a(n), b(n)\in\mathbb{C}$ and $\Gamma(a;x):=\int_{x}^\infty e^{-t}t^{a-1}dt.$
	We refer the two sums in \eqref{inf} as the \emph{holomorphic} and \emph{non-holomorphic} parts of $f$,
	respectively. We denote by \( \widetilde{\mathcal{M}}_{k/2}(N) \) the \(\mathbb{C}\)-vector space of weak Maass forms of weight \( k/2 \) on \( \Gamma_1(N) \).	
	The readers are referred to \cite{bb,o2} for more details.
	
Suppose that \( 0 < a < c \) are integers, and let \(\ell_c := \mathrm{lcm}(2c^2, 24), \widetilde{\ell}_c := \ell_c / 24, f_c := \frac{2c}{\gcd(c,6)} \). Define
\begin{align*}
	\Theta\left(\frac{a}{c}; \tau\right) :&= \sum_{m \pmod{f_c}} (-1)^m \sin\left(\frac{a\pi(6m+1)}{c}\right) \cdot \theta\left(6m+1, 6f_c; \frac{\tau}{24}\right),\no
 \intertext{and}
S_1\left(\frac{a}{c}; z\right) :&= \frac{-i \sin\left(\frac{\pi a}{c}\right) \ell_c^{\frac{1}{2}}}{\sqrt{3}}
 \int_{-\bar{z}}^{i\infty} \frac{\Theta\left(\frac{a}{c}; \ell_c \tau\right)}{\sqrt{-i(\tau + z)}}  \mathrm{d}\tau,
 \intertext{where}
 \theta(\alpha, \beta; \tau) :&= \sum_{n \equiv \alpha \pmod{\beta}} n e^{2\pi i \tau n^2}.
\end{align*}
Then Bringmann and Ono proved
\begin{theorem}(\cite[Theorem 1.1]{ann})\label{thann}
Define the group \(\Gamma_c\) by
\[
\Gamma_c := \left\{
\begin{pmatrix}
	1 & 1 \\
	0 & 1
\end{pmatrix}, 
\begin{pmatrix}
	1 & 0 \\
	\ell_c^2 & 1
\end{pmatrix}
\right\}.
\] Then \(D\left(\frac{a}{c}; z\right) := -S_1\left(\frac{a}{c}; z\right) + q^{-\frac{\ell_c}{24}} R(\zeta_c^a; q^{\ell_c})\) is a weak Maass form of weight \(1/2\) on \(\Gamma_c\).
\end{theorem}

This result is  recently strengthened by Garvan \cite{gar}. Moreover, he also obtained the following generalizing of \eqref{idlo}.
	 For primes $p>3$ and
 \( 1 \leq a \leq \frac{1}{2}(p - 1) \), let
	 \begin{align*}
	 \Phi_{p,a}(q) := 
	 \left\{
	 \begin{aligned}
	 	&	\sum_{n=0}^{\infty} \frac{q^{pn^2}}{(q^a; q^p)_{n+1}(q^{p-a}; q^p)_n}, && \text{if } 0 < 6a < p, \\
	 	&	-1 + \sum_{n=0}^{\infty} \frac{q^{pn^2}}{(q^a; q^p)_{n+1}(q^{p-a}; q^p)_n}, && \text{if } p < 6a < 3p,
	 \end{aligned}
	 \right.
	 \end{align*}
	 and
	 \begin{align*}
	 \mathcal{R}_p(z) := q^{-\frac{1}{24}} R(\zeta_p, q) - \chi_{12}(p)& \sum_{a=1}^{\frac{1}{2}(p-1)} (-1)^a \Bigg( \zeta_p^{3a+\frac{1}{2}(p+1)} + \zeta_p^{-3a-\frac{1}{2}(p+1)} \\&\qquad\qquad- \zeta_p^{3a+\frac{1}{2}(p-1)} - \zeta_p^{-3a-\frac{1}{2}(p-1)} \Bigg) q^{\frac{a}{2}(p-3a)-\frac{p^2}{24}} \Phi_{p,a}(q^p),
	 \end{align*}
	 where
	 \[
	 \chi_{12}(n) := \left( \frac{12}{n} \right) = 
	 \begin{cases} 
	 	1 & \text{if } n \equiv \pm 1 \pmod{12}, \\
	 	-1 & \text{if } n \equiv \pm 5 \pmod{12}, \\
	 	0 & \text{otherwise}.
	 \end{cases}
	 \]
	 Then there holds
	 \begin{theorem}(\cite[Theorem 1.2]{gar})\label{thgar}
	 	Let $\eta(\tau)$ denote the Dedekind's eta-function which is defined by
	 	$$\eta(\tau):=q^{\frac{1}{24}}(q;q)_\infty.$$
Then the function
	 	\[
	 	\eta(p^2z) \mathcal{R}_p(z)
	 	\]
	 	is a weakly holomorphic modular form of weight 1 on the group \( \Gamma_0(p^2) \cap \Gamma_1(p) \).
	 \end{theorem}
	 
	 For more works on the modularity of the rank type generating functions, the reader
	 is referred to \cite{ah,lo1,lo2,fq,cjs1,cjs2,mao}. In this paper, we consider modular properties of the generating functions of odd ranks, extending the works of Bringmann-Ono and Garvan.
	 
	 To provide partition-theoretic interpretations to Watson's the third order mock theta function $\omega(q)$ \cite{wat}, Andrews \cite{an} studied the odd Durfee symbols.
	 \allowdisplaybreaks
	 \begin{definition}\label{defodd}
	 	An odd Durfee symbol of $n$ is a two-rowed array with a subscript of the form
	 	\begin{align*}
	 		\begin{pmatrix}
	 			a_{1}& a_{2}& \cdots& a_{s}\\
	 			b_{1}& b_{2}& \cdots& b_{t}
	 		\end{pmatrix}_{D}
	 	\end{align*}
	 	wherein all entries are odd numbers such that
	 	\begin{enumerate}[(1)]
	 		\item $2D+1\geq a_{1}\geq a_{2}\geq \cdots \geq a_{s}\ge 0$;
	 		\item $2D+1\geq b_{1}\geq b_{2}\geq \cdots \geq b_{t}\ge 0$;		
	 		\item $n=\sum_{i=1}^{s}a_{i}+\sum_{j=1}^{t}b_{j}+2D^{2}+2D+1$.
	 	\end{enumerate}
	 \end{definition}
	 We remark that
	 Definition \ref{defodd} is made by Ji \cite{ji} and is equivalent to Andrews' original definition.
	 Andrews also defined the odd rank of an odd Durfee symbol
	 to be the number of entries in the top row minus the number of entries in the bottom row. 
	 Motivated by Andrews' work, odd ranks are widely studied.
	 Let $N^0(m,n)$ denote the number of odd Durfee symbols of $n$ with odd rank $m$ and
	 	\begin{align}
	 	N^0(m,k,n):=\sum_{t\equiv m\pmod{k}}	N^0(t,n).\no
	 \end{align}
	 
	 Wang \cite{wang} established identities between odd ranks modulo $8$. For example, he obtained the following analogs of \eqref{1id5}:
	 \begin{align*}
	 	\sum_{n=0}^{\infty}(N^0(0,8,8n+1)-N^0(4,8,8n+1))q^n&=\frac{J_2^4}{J_1^2J_4},\\
	 	 	\sum_{n=0}^{\infty}(N^0(1,8,8n+2)-N^0(3,8,8n+2))q^n&=\frac{J_2J_4}{J_1}.
	 \end{align*}
	 Moreover, Wang also proved that, for $n\geq0$,
	 \begin{align}
	 	N^0(0,8 ,8 n+r)&=N^0(4,8 , 8 n+r), & r \in\{5,7\} \no
	 	\intertext{and}
	 	N^0(1,8 , 8 n+r)&=N^0(3,8 , 8 n+r), & r \in\{4,6\}.\no
	 \end{align}
	 
	 Cui and Gu \cite{cui} studied the generating functions of odd ranks modulo $3, 6$ and showed that 
	 \begin{align*}
	 	\sum_{n=0}^{\infty}(N^0(0,3,n)-N^0(1,3,n))q^n&=q\sum_{n=0}^{\infty}\frac{q^{2n(n+1)}(q;q^{2})_{n+1}}{(q^3;q^{6})_{n+1}},
	 	\\
	 	\sum_{n=0}^{\infty}(N^0(0,6,n)-N^0(3,6,n))q^n&=q\frac{(q^3;q^{3})^2_{\infty}(q^{12};q^{12})^2_{\infty}}{(q^2;q^2)_{\infty}(q^6;q^{6})^2_{\infty}}.
	 \end{align*}
	 More recently, Xia \cite{xia} considered odd ranks modulo $12$ and found 
	 \begin{align*}
	\sum_{n=0}^{\infty}(N^0(0,12,4n+3)-N^0(6,12,4n+3))q^n&=\frac{J_6^4}{J_2J^2_3},\\
		\sum_{n=0}^{\infty}(N^0(0,12,4n+1)-N^0(6,12,4n+1))q^n&=\frac{J_2^3}{J^2_1}.
	 \end{align*}
	
%
In this paper, we study transformation properties of the odd rank generating functions
	\begin{align*}
		R^0(z ; q)&:=\sum_{n=1}^{\infty} \sum_{m \in \mathbb{Z}} N^0(m, n) z^m q^n\\
			&= \sum_{n=0}^\infty \frac{q^{2n(n+1)+1}}{(zq;q^2)_{n+1}(z^{-1}q;q^2)_{n+1}}. \qquad\qquad\textrm{(\cite[Eq. (8.3)]{an})}
	\end{align*}
	
For integers $0<a<c$ and $(a,c)=1$,	
we define
	\begin{align}
		\mathscr{R}^0(a,c; \tau)&=q^{-\frac{1}{3}}	R^0(\zeta_c^a ; q)\label{defr0}
		\intertext{and}
		\widehat{\mathscr{R}^0}(a,c ; \tau)&=	\mathscr{R}^0(a,c; \tau)-\frac{\sqrt{6}}{2}\times\Bigg(e^{-\frac{\pi i}{6}}\int_{-\overline{\tau}}^{i \infty} \frac{g_{\frac{2}{3}, \frac{1}{2}-\frac{3a}{c}}(6z)}{\sqrt{-i(z+\tau)}} d z\no\\&\qquad\qquad\qquad\qquad\qquad\qquad+e^{\frac{\pi i}{6}}\int_{-\overline{\tau}}^{i \infty} \frac{g_{\frac{1}{3}, \frac{1}{2}-\frac{3a}{c}}(6z)}{\sqrt{-i(z+\tau)}} d z\Bigg),\label{defhr}
	\end{align}
	where for $\alpha,\beta\in \mathbb{R}$ \[
	g_{\alpha,\beta}(\tau) := \sum_{\nu \in \alpha + \mathbb{Z}} \nu e^{\pi i \nu^2 \tau + 2 \pi i \nu \beta}.
	\]
	Then one of our main results is the following analog of Theorem \ref{thann}:
	\begin{theorem}\label{thm1}
	$\widehat{\mathscr{R}^0}(a,c ; 18\tau)$ is a weak Maass form of weight $1/2$ on $\Gamma_0(864)\cap\Gamma_0(72c^2)\cap\Gamma_1(8)\cap\Gamma_1(4c)$.
	\end{theorem}
The proof of Theorem \ref{thm1} is motivated by works in \cite{cjs1,cjs2,mao} and relies on  Zwegers' works on Appell–Lerch sums \cite{mock,zew1}.
Applying the transformation formations obtained in the proof of Theorem \ref{thm1}, we are able to construct families of weakly modular forms as Garvan did in Theorem \ref{thgar}.
Let
\begin{align}
	& S(r, c ; \tau)=\frac{i^{1-c}q^{c(6r+2-c)-3r^2-2r}}{\left( q^{2 c^2} ; q^{2 c^2}\right)_{\infty}} \sum_{n=-\infty}^{\infty} \frac{(-1)^{n } q^{3 c^2 n^2+5 c^2 n}}{1+(-1)^cq^{6 c^2 n+(2+6r-c) c}}.
\end{align}
Then there holds
\begin{theorem}\label{thm2}
	\begin{enumerate}[(i)]
		\item 
			The function 
		\begin{align}
			\frac{\eta^5\left(2\tau\right)}{\eta^2(\tau)}\left(	\mathscr{R}^0(a,c; \tau)-i^{1-c}q^{-\frac{1}{3}}\sum_{r=0}^{c-1}(-1)^r\left(\zeta_c^{(3r+1)a}+\zeta_c^{-(3r+1)a}\right)S(r, c ; \tau)\right)
		\end{align}
		is a weakly holomorphic modular form of weight $2$ on $\Gamma_0\left([\gcd(c,2)]^212c^2\right)\cap\Gamma_1(6c).$
		
		\item
		If $3\nmid c$, then the function 
		\begin{align}
			\frac{\eta^5\left(2c^2\tau\right)}{\eta^2(c^2\tau)}\left(	\mathscr{R}^0(a,c; \tau)-i^{1-c}q^{-\frac{1}{3}}\sum_{r=0}^{c-1}(-1)^r\left(\zeta_c^{(3r+1)a}+\zeta_c^{-(3r+1)a}\right)S(r, c ; \tau)\right)
		\end{align}
		is a weakly holomorphic modular form of weight $2$ on $\Gamma_0\left([\gcd(c,2)]^212c^2\right)\cap\Gamma_1(6c).$
	\end{enumerate}
\end{theorem}
Comparing to Theorem \ref{thgar}, we do not require $c$ to be a prime number in Theorem \ref{thm2}.
%
%
As an application of part (ii) of Theorem \ref{thm2}, we establish an analog of \eqref{idlo}.
\begin{theorem}\label{thm3}
	We have
	\begin{align}\label{r5}
	R^0(\zeta_5 ; q)&=-2S(3,5;\tau)+(\zeta_5+\zeta_5^4)\left(S(0,5;\tau)-S(1,5;\tau)\right)
		+(\zeta_5^2+\zeta_5^3)\left(S(2,5;\tau)+S(4,5;\tau)\right)
		\no	\\&\quad+\frac{q^5J_{150}}{[q^{15},q^{50};q^{150}]_\infty}-\frac{J_{50}[q^{25};q^{50}]_\infty}{[q^{10};q^{50}]_\infty^2}
		+\frac{qJ_{50}[q^{15};q^{50}]_\infty}{[q^{10};q^{50}]_\infty^2}
		-\frac{q^2J_{50}[q^{15};q^{50}]^2_\infty}{[q^{10},q^{10},q^{25};q^{50}]_\infty}
		\no	\\&\quad-\frac{q^{17}J_{150}}{[q^{50},q^{75};q^{150}]_\infty}
		+\frac{q^3J_{50}[q^5,q^{25};q^{50}]_\infty}{[q^{10},q^{10},q^{15};q^{50}]_\infty}
		-\frac{q^4J_{50}[q^5;q^{50}]_\infty}{[q^{10};q^{50}]^2_\infty}
		\no\\&\quad	+(\zeta_5^2+\zeta_5^3)\Bigg(\frac{q^5J_{150}}{[q^{15},q^{50};q^{150}]_\infty}-\frac{J_{50}[q^{15};q^{50}]^2_\infty}{[q^{5},q^{10},q^{20};q^{50}]_\infty}-q^6\frac{J_{50}[q^{5};q^{50}]_\infty}{[q^{10},q^{20};q^{50}]_\infty}
		\no\\&\qquad\qquad\qquad\quad-q^2\frac{J_{50}[q^{25};q^{50}]_\infty}{[q^{10},q^{20};q^{50}]_\infty}
		+q^3\frac{J_{50}[q^{15};q^{50}]_\infty}{[q^{10},q^{20};q^{50}]_\infty}+q^9\frac{J_{50}[q^{5};q^{50}]^2_\infty}{[q^{10},q^{15},q^{20};q^{50}]_\infty}
		\no\\&\qquad\qquad\qquad\quad+\frac{q^{14}J_{150}}{[q^{45},q^{50};q^{150}]_\infty}\Bigg).
	\end{align}
\end{theorem}

Applying Theorem \ref{thm3}, we obtain the following $5$-dissections.
\begin{corollary}\label{coro}
	We have 
	\begin{align}
		&\sum_{n=1}^{\infty}\left(N^0(0,5 ; n)-N^0(1,5 ; n)\right) q^n\no
		\\&=S(1,5;\tau)-S(0,5;\tau)-2S(3,5;\tau)+\frac{q^5J_{150}}{[q^{15},q^{50};q^{150}]_\infty}-\frac{J_{50}[q^{25};q^{50}]_\infty}{[q^{10};q^{50}]_\infty^2}
		\no	\\&\quad
		+\frac{qJ_{50}[q^{15};q^{50}]_\infty}{[q^{10};q^{50}]_\infty^2}
		-\frac{q^2J_{50}[q^{15};q^{50}]^2_\infty}{[q^{10},q^{10},q^{25};q^{50}]_\infty}
-\frac{q^{17}J_{150}}{[q^{50},q^{75};q^{150}]_\infty}		\no	\\&\quad
		+\frac{q^3J_{50}[q^5,q^{25};q^{50}]_\infty}{[q^{10},q^{10},q^{15};q^{50}]_\infty}
		-\frac{q^4J_{50}[q^5;q^{50}]_\infty}{[q^{10};q^{50}]^2_\infty}\label{co01}
	\intertext{and}
		&\sum_{n=1}^{\infty}\left(N^0(1,5 ; n)-N^0(2,5 ; n)\right) q^n\no
	\\&=S(0,5;\tau)-S(1,5;\tau)-S(2,5;\tau)-S(4,5;\tau)	-\frac{q^5J_{150}}{[q^{15},q^{50};q^{150}]_\infty}
		\no\\&\quad+\frac{J_{50}[q^{15};q^{50}]^2_\infty}{[q^{5},q^{10},q^{20};q^{50}]_\infty}+q^6\frac{J_{50}[q^{5};q^{50}]_\infty}{[q^{10},q^{20};q^{50}]_\infty}+q^2\frac{J_{50}[q^{25};q^{50}]_\infty}{[q^{10},q^{20};q^{50}]_\infty}	
		\no\\&\quad 	-q^3\frac{J_{50}[q^{15};q^{50}]_\infty}{[q^{10},q^{20};q^{50}]_\infty}-q^9\frac{J_{50}[q^{5};q^{50}]^2_\infty}{[q^{10},q^{15},q^{20};q^{50}]_\infty}-\frac{q^{14}J_{150}}{[q^{45},q^{50};q^{150}]_\infty}.
		\label{co12}
		\end{align}
\end{corollary}
Analogs of \eqref{1id5} follows from Corollary \ref{coro} immediately. For example, extracting only terms with $q^n$ where $n\equiv 1, 2, 3\pmod{5}$ in \eqref{co12} yields
 \begin{align}
	\sum_{n=0}^{\infty}(N^0(1,5,5n+1)-N^0(2,5,5n+1))q^n&=q\frac{J_{10}[q;q^{10}]_\infty}{[q^{2},q^{4};q^{10}]_\infty},\no\\
		\sum_{n=0}^{\infty}(N^0(1,5,5n+2)-N^0(2,5,5n+2))q^n&=\frac{J_{10}[q^5;q^{10}]_\infty}{[q^{2},q^{4};q^{10}]_\infty},\no
		\intertext{and}
			\sum_{n=0}^{\infty}(N^0(1,5,5n+3)-N^0(2,5,5n+3))q^n&=-\frac{J_{10}[q^3;q^{10}]_\infty}{[q^{2},q^{4};q^{10}]_\infty}.\no
\end{align}

This article is organized as follows. In Section \ref{s2}, we first recall some results on the Dedekind's eta function and Zwegers' works on Appell–Lerch sums. Then we study the modular completion of the odd rank generating 
using Zwegers' results. After establishing transformations formulas of the concerned functions, we prove Theorems \ref{thm1} and \ref{thm2} in Section \ref{s3}.
In Section \ref{s4}, we apply the valence formula to show Theorem \ref{thm3}. In this proof, we also need Frey and Garvan’s
MAPLE package \cite{maple} to study the cusp behaviors of some modular functions. Lastly, using Theorem \ref{thm3}, we obtain Corollary \ref{coro}.

\section{Preliminaries}\label{s2}
\subsection{Dedekind’s eta-function} 
As a building block of many classical modular forms, the Dedekind's eta-function plays an important role in the theory of modular forms.
The transformation formula of $\eta(\tau)$ under $\mathrm{SL}_2(\mathbb{Z})$ is given by
	\begin{align}
	\eta(B \tau)=\nu(B) \sqrt{\gamma \tau+\delta} \eta(\tau),   \label{treta}
\end{align}
where $B=\left(\begin{array}{cc}\alpha & \beta \\ \gamma & \delta\end{array}\right) \in \mathrm{SL}_2(\mathbb{Z})$ and
	$$
\nu(B)= \begin{cases}\left(\frac{\delta}{|\gamma|}\right) \exp \left(\frac{\pi i}{12}\left((\alpha+\delta) \gamma-\beta \delta\left(\gamma^2-1\right)-3 \gamma\right)\right) & \text { if } \gamma \equiv 1 \pmod{2}, \\ \left(\frac{\gamma}{\delta}\right) \exp \left(\frac{\pi i}{12}\left((\alpha+\delta) \gamma-\beta \delta\left(\gamma^2-1\right)+3 \delta-3-3 \gamma \delta\right)\right) & \text { if } \delta \equiv 1 \pmod{2}.\end{cases}
$$	

%
 Moreover, we also have the following lemma. Define ${ }^m B:=\left(\begin{array}{cc}\alpha & m \beta \\ \gamma / m & \delta\end{array}\right)$ for $B=\left(\begin{array}{ll}
	\alpha & \beta \\
	\gamma &\delta
\end{array}\right) \in \Gamma_0(m)$.
 \begin{lemma}
 	We have
 		\begin{align}
 		\frac{\nu^2\left(B\right)}{\nu^4\left({}^{2}B\right)}= (-1)^{\frac{\alpha-1}{2}(1+\beta)} e^{-\frac{ \pi i\beta}{2}}\label{etaj}
 	\end{align}
 	with $B=\left(\begin{array}{cc}\alpha & \beta \\ \gamma & \delta\end{array}\right) \in\Gamma_0(2).$
%
%
\end{lemma}
\begin{proof}
%
	For $B=\left(\begin{array}{ll}
		\alpha & \beta \\
		\gamma &\delta
	\end{array}\right) \in \Gamma_0(2)$,
we have
	\begin{align}
		\frac{\eta\left(2B\tau\right)}{\eta^2(B\tau)}\no&= \frac{1}{\sqrt{\gamma \tau+\delta}} \frac{\nu\left({}^{2}B\right)\eta\left(2\tau\right)}{\nu^2\left(B\right)\eta^2(\tau)}
		\\&= (-1)^{\frac{\alpha-1}{2}+\beta}\no i^{-\alpha \beta}\nu^{-3}\left({}^2B\right)\frac{1}{\sqrt{\gamma \tau+\delta}} \frac{\eta\left(2\tau\right)}{\eta^2(\tau)},\no
	\end{align}
where the second equality follows from \cite[Proposition 2.1]{cs}.
Thus
	\begin{align}
\frac{\nu^2\left(B\right)}{\nu^4\left({}^{2}B\right)}&= (-1)^{\frac{\alpha-1}{2}+\beta}\no i^{\alpha \beta}
\\&= (-1)^{\frac{\alpha-1}{2}+\beta}\no i^{(\alpha-1+1) \beta}
\end{align}
which gives \eqref{etaj}.
\end{proof}

\subsection{Modularity of Appell-Lerch sums} We recall Zwegers' works \cite{mock,zew1} on the modular completion of the Appell-Lerch sums.
Let
	\begin{align}
	\vartheta(z ; \tau):&=\sum_{\nu \in \frac{1}{2}+\mathbf{Z}} e^{\pi i \nu^{2} \tau+2 \pi i \nu\left(z+\frac{1}{2}\right)}\nonumber
	\\&=-i q^{\frac{1}{8}} e^{-\pi i z} \prod_{n=1}^{\infty}\left(1-q^{n}\right)\left(1-e^{2 \pi i z} q^{n-1}\right)\left(1-e^{-2 \pi i z} q^{n}\right).\label{theta}\\
	\mu(u, v ; \tau)&:=\frac{e^{\pi i  u }}{	\vartheta(v;\tau)} \sum_{n \in \mathbb{Z}} \frac{(-1)^{ n } e^{\pi i n (n+1)\tau+2\pi i n v } }{1- e^{2\pi i n\tau+2\pi iu}},\label{defapp}\\
	A_{\ell}(u, v ; \tau)&:=e^{\pi i \ell u } \sum_{n \in \mathbb{Z}} \frac{(-1)^{\ell n } e^{\pi i\ell n (n+1)\tau+2\pi i n v } }{1- e^{2\pi i n\tau+2\pi iu}}.\label{defapp}
\end{align}
The $\mu$-functions and $A_\ell$-functions can be completed to non-holomorphic modular forms.
To describe the completion, we need the nonholomorphic function $R(u;\tau)$ which is defined by
$$
\begin{aligned}
	R(u ; \tau):= & \sum_{\nu \in \frac{1}{2}+\mathbb{Z}}\{\operatorname{sgn}(v)-E((v+\operatorname{Im}(u) / \operatorname{Im}(\tau)) \sqrt{2 \operatorname{Im}(\tau)})\} \\
	& \times(-1)^{\nu-\frac{1}{2}} q^{-v^2 / 2} e^{-2 \pi i v u}
\end{aligned}
$$
with
$$
\begin{aligned}
	& E(z):=2 \int_0^z e^{-\pi u^2} \mathrm{~d} u=\operatorname{sgn}(z)\left(1-\beta\left(z^2\right)\right) \quad(z \in \mathbb{R}) \\
	& \beta(x):=\int_x^{\infty} u^{-\frac{1}{2}} e^{-\pi u} \mathrm{~d} u \quad\left(x \in \mathbb{R}_{\geq 0}\right).
\end{aligned}
$$
We remark that, by \cite[Eq. (2.17)]{cjs1}, we have for \( a \in \mathbb{Q} \),
\begin{align}
R(a\tau - b;\tau) = p(\tau) + O(y^{-\frac{1}{2}}e^{-\pi a^2 y - \epsilon y}),\label{gth}
\end{align}
as \( y \to \infty \), where \( \epsilon > 0 \) and \( p(\tau) \) is some rational function of a fractional power of \( q \).

Define
\begin{align}
	\widehat{\mu}(u, v ; \tau):= & \mu(u, v ; \tau)+\frac{i}{2}R(u-v;\tau)\label{cnu},\\
	\widehat{A}_l(u, v ; \tau):= & A_l(u, v ; \tau) \no\\
	& +\frac{i}{2} \sum_{k=0}^{l-1} e^{2 \pi i k u} \vartheta(v+k \tau+(l-1) / 2 ; l \tau) R(l u-v-k \tau-(l-1) / 2 ; l \tau).
\end{align}
	Then Zwegers \cite{mock,zew1} proved that
	\begin{align}
		&	\widehat{A}_{\ell}\left(u+n_{1} \tau+m_{1}, v+n_{2} \tau+m_{2};\tau\right)\nonumber\\ &=(-1)^{\ell\left(n_{1}+m_{1}\right)} e^{2\pi iu(\ell n_{1}-n_{2})} e^{-2\pi ivn_{1}} q^{\ell n_{1}^{2} / 2-n_{1} n_{2}} \widehat{A}_{\ell}(u, v;\tau),\label{appel}\\
	&	\widehat{\mu}(u + n_1\tau + m_1, v + n_2\tau + m_2;\tau) \no\\&= (-1)^{n_1+m_1+n_2+m_2} e^{\pi i\tau(n_1-n_2)^2 + 2\pi i (n_1-n_2)(u-v)} \widehat{\mu}(u,v;\tau),
		\intertext{for all $n_1, n_2, m_1, m_2\in\mathbb{Z}$ and}
		&	\widehat{A}_{\ell}\left(\frac{u}{\gamma \tau+\delta}, \frac{v}{\gamma \tau+\delta} ; \frac{\alpha \tau+\beta}{\gamma \tau+\beta}\right)\nonumber\\&=(\gamma \tau+\delta) e^{\pi i \gamma\left(-\ell u^{2}+2 u v\right) /(\gamma \tau+\delta)} \widehat{A}_{\ell}(u, v ; \tau),\label{appmo}\\
	&	\widehat{\mu}\left(\frac{u}{\gamma \tau+\delta}, \frac{v}{\gamma \tau+\delta} ; \frac{\alpha \tau+\beta}{\gamma \tau+\beta}\right)\no\\&= \nu^{-3}(B) \sqrt{\gamma \tau+\beta} \, e^{-\pi i\gamma(u-v)^2 / (\gamma \tau+\beta)} \widehat{\mu}(u,v; \tau)\label{trmu}
		\intertext{for $B=\left(\begin{array}{ll}
				\alpha & \beta \\
				\gamma &\delta
			\end{array}\right) \in \mathrm{SL}_{2}(\mathbb{Z})$. }\nonumber
	\end{align}

In the next proposition, we rewrite $	R^0(\zeta ; q)$ in terms of completed Appell-Lerch sums. 
\begin{proposition}\label{pro828}
	We have
	\begin{align}
		&	\widehat{\mathscr{R}^0}(a,c ; \tau)\no\\&=\frac{\zeta_{2c}^{-3a}q^{\frac{-3}{4}}}{\eta(2\tau)}\widehat{A_3}\left(\tau+\frac{a}{c},0;2\tau\right)\label{ra3}\\
		&=-i q^{-\frac{1}{12}}\zeta_{2c}^a\widehat{	\mu}\left(3\tau-\frac{a}{c}, \frac{2a}{c}+2\tau ; 6\tau\right)
		-i q^{-\frac{1}{12}}\zeta_{2c}^{-a}\widehat{	\mu}\left(\tau-\frac{a}{c}, \frac{2a}{c}+2\tau ; 6\tau\right).\label{ru1}
	\end{align}
\end{proposition}
\begin{proof}
	Recall \cite[Eq.(8.5)]{an}:
	\begin{align*}
		R^0(\zeta ; q)&=\frac{1}{\left(q^2 ; q^2\right)_{\infty}} \sum_{n=-\infty}^{\infty} \frac{(-1)^n q^{3 n^2+3 n+1}}{1-\zeta q^{2 n+1}}.
	\end{align*}
	Then \begin{align*}
		R^0(\zeta ; q)&=\frac{1}{q^{\frac{5}{12}}\zeta^{\frac{3}{2}}\eta(2\tau)} A_3(\tau+z,0;2\tau),
	\end{align*}
	which together with \eqref{defr0} gives
	\begin{align*}
		\mathscr{R}^0(a,c; \tau)&=\frac{\zeta_{2c}^{-3a}q^{\frac{-3}{4}}}{\eta(2\tau)}A_3\left(\tau+\frac{a}{c},0;2\tau\right).
	\end{align*}
	We also	note that
	\begin{align}
		\widehat{A_3}\left(\tau+\frac{a}{c},0;2\tau\right)&=	A_3\left(\tau+\frac{a}{c},0;2\tau\right)	\no	\\&\quad+\frac{i}{2}\sum_{k=0}^2e^{2k\pi i\left(\tau+\frac{a}{c}\right)}\vartheta(2k\tau+1;6\tau)
		R\left(3\left(\tau+\frac{a}{c}\right)-2k\tau-1;6\tau\right)
		\no\\&=A_3\left(\tau+\frac{a}{c},0;2\tau\right)+\frac{i}{2}e^{2\pi i\left(\tau+\frac{a}{c}\right)}\vartheta(2\tau;6\tau)
		R\left(\tau+\frac{3a}{c};6\tau\right)
		\no	\\&\quad	+\frac{i}{2}e^{2\pi i\left(2\tau+\frac{2a}{c}\right)}\vartheta(4\tau;6\tau)
		R\left(\frac{3a}{c}-\tau;6\tau\right).\no
	\end{align}
	This together with \eqref{theta} implies
	\begin{align}
		&\frac{\zeta_{2c}^{-3a}q^{\frac{-3}{4}}}{\eta(2\tau)}\widehat{A_3}\left(\tau+\frac{a}{c},0;2\tau\right)
		\no\\&=	\mathscr{R}^0(a,c; \tau)+\frac{q^{-\frac{1}{12}}}{2}\left(\zeta_{2c}^{-a}	R\left(\tau+\frac{3a}{c};6\tau\right)+\zeta_{2c}^{a}	R\left(\frac{3a}{c}-\tau;6\tau\right)\right)\label{5132039}.
	\end{align}
	Apply \cite[Theorem 1.16]{mock} to obtain
	\begin{align}
		R\left(\pm\tau+\frac{3a}{c};6\tau\right)\no&=	R\left(\pm\frac{1}{6}(6\tau)-\frac{-3a}{c};6\tau\right)
		\no\\&=-e^{\pi i\left(\frac{\tau}{6}\pm\left(\frac{a}{c}-\frac{1}{6}\right)\right)}\int_{-\overline{6\tau}}^{i \infty} \frac{g_{\pm\frac{1}{6}+\frac{1}{2}, \frac{1}{2}-\frac{3a}{c}}(z)}{\sqrt{-i(z+6\tau)}} d z
		\no\\&=-\sqrt{6}q^{\frac{1}{12}}\zeta_{2c}^{\pm a}e^{\mp\frac{\pi i}{6}}\int_{-\overline{\tau}}^{i \infty} \frac{g_{\pm\frac{1}{6}+\frac{1}{2}, \frac{1}{2}-\frac{3a}{c}}(6z)}{\sqrt{-i(z+\tau)}} d z.\no
	\end{align}
	It follows 
	\begin{align}
		&\frac{\zeta_{2c}^{-3a}q^{\frac{-3}{4}}}{\eta(2\tau)}\widehat{A_3}\left(\tau+\frac{a}{c},0;2\tau\right)
		\no\\&=		\mathscr{R}^0(a,c; \tau)
		-\frac{\sqrt{6}}{2}\times\left(e^{-\frac{\pi i}{6}}\int_{-\overline{\tau}}^{i \infty} \frac{g_{\frac{2}{3}, \frac{1}{2}-\frac{3a}{c}}(6z)}{\sqrt{-i(z+\tau)}} d z+e^{\frac{\pi i}{6}}\int_{-\overline{\tau}}^{i \infty} \frac{g_{\frac{1}{3}, \frac{1}{2}-\frac{3a}{c}}(6z)}{\sqrt{-i(z+\tau)}} d z\right)
		\no\\&=	\widehat{\mathscr{R}^0}(a,c ; \tau).\no
	\end{align}
	This proves \eqref{ra3}.

	Next, we consider \eqref{ru1}.
	Define 
	\begin{align*}
		g(x;q)& := \sum_{n=0}^\infty \frac{q^{n(n+1)}}{\left( x;q \right)_{n+1} \left( q/x;q \right)_{n+1}},
		\intertext{and}
		m(x,q,z) &:= \frac{-z}{(z,q/z,q;q)} \sum_{r=-\infty}^\infty \frac{(-1)^{r} q^{\frac{r(r+1)}{2}} z^{r}}{1 - q^{r} x z}.
	\end{align*}
	Then \cite[Eq. (4.3)]{hekce} gives
	\begin{align}
		g(x;q) = -x^{-1} m(q^2 x^{-3},q^3, x^2) -x^{-2} m(qx^{-3},q^3 , x^2).\no
	\end{align}
	By
	\cite[Eq. (2.4)]{wang}, we have
	\begin{align}
		R^0(z;q)=qg(zq;q^2),\no
	\end{align}
	It follows 
	\begin{align}
		R^0(z;q)&=-z^{-1}m(q/z^3,q^6,z^2q^2)-z^{-2}q^{-1}m(1/(qz^3),q^6,z^2q^2)\no
		\\	&=-iq^{1/4}z^{1/2}\mu(3\tau-u, 2u+2\tau ; 6\tau)-iq^{1/4}z^{-1/2}\mu(\tau-u, 2u+2\tau ; 6\tau)\no
	\end{align}
	with $z=e^{2\pi i u}$.
	Thus
	\begin{align}
		\widehat{\mathscr{R}^0}(a,c ; \tau)&=q^{-\frac{1}{3}}	R^0(\zeta_c^a ; q)+\frac{q^{-\frac{1}{12}}}{2}\left(\zeta_{2c}^{-a}	R\left(\tau+\frac{3a}{c};6\tau\right)+\zeta_{2c}^{a}	R\left(\frac{3a}{c}-\tau;6\tau\right)\right)\no\\
		&=-i q^{-\frac{1}{12}}\zeta_{2c}^a\left[	\mu\left(3\tau-\frac{a}{c}, \frac{2a}{c}+2\tau ; 6\tau\right)+\frac{i}{2}R\left(\frac{3a}{c}-\tau;6\tau\right)\right]\no
		\\&\quad-i q^{-\frac{1}{12}}\zeta_{2c}^{-a}\left[	\mu\left(\tau-\frac{a}{c}, \frac{2a}{c}+2\tau ; 6\tau\right)+\frac{i}{2}	R\left(\tau+\frac{3a}{c};6\tau\right)\right]
		\no\\
		&=-i q^{-\frac{1}{12}}\zeta_{2c}^a\widehat{	\mu}\left(3\tau-\frac{a}{c}, \frac{2a}{c}+2\tau ; 6\tau\right)\no
		-i q^{-\frac{1}{12}}\zeta_{2c}^{-a}\widehat{	\mu}\left(\tau-\frac{a}{c}, \frac{2a}{c}+2\tau ; 6\tau\right),\no
	\end{align}
	where the last equality follows from \eqref{cnu} and 
	\begin{align}
		R(-u;\tau)=R(u;\tau). \qquad\qquad\textrm{(see Proposition 1.9 of \cite{mock})}
	\end{align}
	This proves \eqref{ru1}.
	%
\end{proof}
Applying Proposition \ref{pro828} to obtain
\begin{proposition}\label{nohr}
	The non-holomorphic part of $\widehat{\mathscr{R}^0}(a,c ; \tau)$ equals to
	\begin{align}
		&	\frac{1 }{ 2\sqrt{\pi}} \sum_{n=-\infty}^{\infty}(-1)^n \left(\zeta_c^{-(3  n+2)a}+\zeta_c^{(3  n+2)a} \right)q^{-3 n^2-4 n-\frac{4}{3}} \operatorname{sgn}\left(n+\frac{2}{3}\right) \Gamma\left(\frac{1}{2}, 12 \pi\left(n+\frac{2}{3}\right)^2 y\right)
		\no	\\&	=
		\frac{i^{1-c}}{2}\sum_{r=0}^{c-1}(-1)^r\left(\zeta_c^{(3r+1)a}+\zeta_c^{-(3r+1)a}\right) q^{-\frac{(3c-2-6r)^2}{12}} 
		\no\\&\qquad\qquad \qquad\times R\left(c (6r+2-3c)\tau-\frac{c-1}{2} ;6c^2 \tau\right)\label{nonholo}.
	\end{align}
\end{proposition}
\begin{proof}
	Note that
	$$
	\int_{-\overline{\tau}}^{i \infty} \frac{e^{2 b \pi i z}}{\sqrt{-i(\tau+z)}} d z=\int_{2 i y}^{i \infty} \frac{e^{2 b \pi i(z-\tau)}}{\sqrt{-i z}} d z=i q^{-b} \int_{2 y}^{\infty} \frac{e^{-2 b \pi z}}{\sqrt{z}} d z=\frac{i q^{-b}}{\sqrt{2 \pi b}} \Gamma\left(\frac{1}{2} ; 4 \pi by\right) 
	$$
	with $b>0.$
	It follows
	\begin{align}
		& \int_{-\overline{\tau}}^{i \infty} \frac{g_{\frac{2}{3}, \frac{1}{2}-\frac{3 a}{c}}(6z )}{\sqrt{-i(z+\tau)}} d z\no
		\\	&=\sum_{n=-\infty}^{\infty}\left(n+\frac{2}{3}\right) e^{2 \pi i\left(n+\frac{2}{3}\right)\left(\frac{1}{2}-\frac{3 a}{c}\right)}\int_{-\overline{\tau}}^{i \infty} \frac{e^{\pi i\left(n+\frac{2}{3}\right)^2 6z}}{\sqrt{-i(z+\tau)}} d z\no \\
		& =\frac{i e^{\frac{2 \pi i}{3}} }{ \sqrt{6\pi}} \sum_{n=-\infty}^{\infty}(-1)^n \zeta_c^{-(3  n+2)a} q^{-3 n^2-4 n-\frac{4}{3}} \operatorname{sgn}\left(n+\frac{2}{3}\right) \Gamma\left(\frac{1}{2}, 12 \pi\left(n+\frac{2}{3}\right)^2 y\right).\label{g1}
	\end{align}
	Similarly, we have	
	\begin{align}
		&\int_{-\overline{\tau}}^{i \infty} \frac{g_{\frac{1}{3}, \frac{1}{2}-\frac{3a}{c}}(6z)}{\sqrt{-i(z+\tau)}} d z\no
		\\& =\frac{i e^{\frac{ \pi i}{3}} \zeta_{ c}^{- a}}{ \sqrt{6\pi}} \sum_{n=-\infty}^{\infty}(-1)^n \zeta_c^{-3 a n} q^{-3 n^2-2 n-\frac{1}{3}} \operatorname{sgn}\left(n+\frac{1}{3}\right) \Gamma\left(\frac{1}{2}, 12 \pi\left(n+\frac{1}{3}\right)^2 y\right)\no\\
		& =\frac{i e^{\frac{ \pi i}{3}} }{ \sqrt{6\pi}} \sum_{n=-\infty}^{\infty}(-1)^n \zeta_c^{(3  n+2)a} q^{-3 n^2-4 n-\frac{4}{3}} \operatorname{sgn}\left(n+\frac{2}{3}\right) \Gamma\left(\frac{1}{2}, 12 \pi\left(n+\frac{2}{3}\right)^2 y\right)\label{g2}.
	\end{align}
	Then we deduce from \eqref{defhr}, \eqref{g1} and \eqref{g2} that, the non-holomorphic part of $\widehat{\mathscr{R}^0}(a,c ; \tau)$ equals to the left side of \eqref{nonholo}.
	
	By \eqref{5132039}, we find that the non-holomorphic part of $\widehat{\mathscr{R}^0}(a,c ; \tau)$ is
	\begin{align}
		\frac{q^{-\frac{1}{12}}}{2}\left(\zeta_{2c}^{-a}	R\left(\tau+\frac{3a}{c};6\tau\right)+\zeta_{2c}^{a}	R\left(\frac{3a}{c}-\tau;6\tau\right)\right).\label{rno}
	\end{align}
	Recall the following identity (see \cite[p. 7]{zew1}):
	\begin{align}
		R(u ; \tau /\ell)
		= & \sum_{r=0}^{\ell-1} e^{-\pi i(r-(\ell-1) / 2)^2 \tau / \ell+2 \pi i(r-(\ell-1) / 2)(u+1 / 2)} \no\\
		&\qquad \times R(\ell u-r \tau+(\ell-1)(\tau+1) / 2 ;\ell \tau).\label{idr}
	\end{align}
	Replacing $(\tau,\ell,u)$ by $(6c\tau,c,\tau+\frac{3a}{c})$ in \eqref{idr} to obtain
	\begin{align}
		R\left(\tau+\frac{3a}{c} ; 6\tau\right)
		= & \sum_{r=0}^{c-1} e^{-\pi i(r-(c-1) / 2)^2 6\tau +2 \pi i(r-(c-1) / 2)\left(\tau+\frac{3a}{c}+1 / 2\right)} \no\\
		&\qquad \times R\left(c \left(\tau+\frac{3a}{c}\right)-r 6c\tau+(c-1)(6c\tau+1) / 2 ;6c^2 \tau\right)\no\\
		= &  i^{1-c}\zeta_c^{\frac{3a}{2}}\sum_{r=0}^{c-1} (-1)^r\zeta_c^{3ar} q^{\frac{(2r+1-c)(3c-1-6r)}{4}} \no\\
		&\qquad\qquad\quad \times R\left(c (6r+2-3c)\tau-\frac{c-1}{2} ;6c^2 \tau\right).\label{r1}
	\end{align}
	Similarly, 
	\begin{align}
		R\left(-\tau+\frac{3a}{c} ; 6\tau\right)=&	R\left(\tau-\frac{3a}{c} ; 6\tau\right)
		\no\\
		= &  i^{1-c}\zeta_c^{-\frac{3a}{2}}\sum_{r=0}^{c-1} (-1)^r\zeta_c^{-3ar} q^{\frac{(2r+1-c)(3c-1-6r)}{4}} \no\\
		&\qquad\qquad\quad \times R\left(c (6r+2-3c)\tau-\frac{c-1}{2} ;6c^2 \tau\right).\label{r2}
	\end{align}
	Substituting \eqref{r1} and \eqref{r2} into \eqref{rno} gives the right side of \eqref{nonholo}. This proves Proposition \ref{nohr}.
\end{proof}
As a consequence we have
\begin{corollary}
	Let
	\begin{align*}
		& \widetilde{\mathcal{S}}(r, c ; \tau):=q^{-\frac{1}{3}}S(r, c ; \tau)+\frac{1}{2}q^{-\frac{(3c-2-6r)^2}{12}}  R\left(c (6r+2-3c)\tau-\frac{c-1}{2} ;6c^2 \tau\right) .
	\end{align*}
	Then 
		\begin{align}
		&	\widehat{\mathscr{R}^0}(a,c; \tau)-i^{1-c}\sum_{r=0}^{c-1}(-1)^r\left(\zeta_c^{(3r+1)a}+\zeta_c^{-(3r+1)a}\right) \widetilde{\mathcal{S}}(r, c ; \tau)\no
		\\&=\mathscr{R}^0(a,c; \tau)-i^{1-c}q^{-\frac{1}{3}}\sum_{r=0}^{c-1}(-1)^r\left(\zeta_c^{(3r+1)a}+\zeta_c^{-(3r+1)a}\right)S(r, c ; \tau).\label{holo}
	\end{align}
In particular,
	the function
	\begin{align}
		\widehat{\mathscr{R}^0}(a,c; \tau)-i^{1-c}\sum_{r=0}^{c-1}(-1)^r\left(\zeta_c^{(3r+1)a}+\zeta_c^{-(3r+1)a}\right) \widetilde{\mathcal{S}}(r, c ; \tau)\no
	\end{align}
	is holomorphic in $\tau$ and has at worst poles at the cusps. 
\end{corollary}

\section{Proofs of Theorems \ref{thm1}, \ref{thm2} and Corollary \ref{coro}}\label{s3}
\subsection{Transformation formulas}
We have
\begin{proposition}\label{pro1514}
	For positive integers $k$ and $B=\left(\begin{array}{ll}
		\alpha & \beta \\
		\gamma &\delta
	\end{array}\right)\in\Gamma_0\left(2kc\right)\cap\Gamma_1(c)$, we have
		\begin{align}
		\widehat{\mathscr{R}^0}(a,c ; kB\tau)&=\frac{\sqrt{\gamma\tau+\delta}}{\nu\left({}^{2k}B\right)}	(-1)^{\frac{\alpha-1}{2}(1+k\beta)+\frac{a\gamma}{2kc}+\frac{a^2\gamma(\delta-1)}{2kc^2}}
		\no	\\&\quad\times e^{\frac{-\pi i}{2}\left[k\beta-\frac{3\gamma a^2}{kc^2}\right]}			\widehat{\mathscr{R}^0}(a,c ; k\tau).\label{kr}
	\end{align}
\end{proposition}
\begin{proof}
	Let
	\begin{align*}
		F_k(a,c;\tau)&:=\widehat{A_3}\left(k\tau+\frac{a}{c},0;2k\tau\right).
	\end{align*}
Then
\begin{align}
	\widehat{\mathscr{R}^0}(a,c ;k \tau)=\frac{\zeta_{2c}^{-3a}q^{\frac{-3k}{4}}}{\eta(2k\tau)}	F_k(a,c;\tau)\label{ptrr}.
\end{align}
Apply \eqref{appmo} to obtain
\begin{align}
&	F_k\left(a,c;B\tau\right)\no\\
&=	\widehat{A_3}\left(kB\tau+\frac{a}{c},0;2kB\tau\right)\no\\&=	\widehat{A_3}\left(\frac{k(\alpha\tau+\beta)}{\gamma\tau+\delta}+\frac{a(\gamma\tau+\delta)}{c(\gamma\tau+\delta)},0;\frac{\alpha(2k\tau)+2k\beta}{\frac{\gamma}{2k}(2k\tau)+\delta}\right)\no\\
	&=(\gamma\tau+\delta)e^{\frac{\pi i\gamma}{2k(\gamma\tau+\delta)}\left\{-3\left[k(\alpha\tau+\beta)+\frac{a(\gamma\tau+\delta)}{c}\right]^2\right\}}\times	\widehat{A_3}\left(k(\alpha\tau+\beta)+\frac{a(\gamma\tau+\delta)}{c},0;2k\tau\right) \no\\
	&=(\gamma\tau+\delta)e^{\frac{-3\pi i}{2}\left\{\frac{k\gamma(\alpha\tau+\beta)^2}{(\gamma\tau+\delta)}+\frac{2\gamma a(\alpha\tau+\beta)}{c}+\frac{\gamma a^2(\gamma\tau+\delta)}{kc^2}\right\}}\times	\widehat{A_3}\left(k(\alpha\tau+\beta)+\frac{a(\gamma\tau+\delta)}{c},0;2\tau\right)\no\\
	&=(\gamma\tau+\delta)
	e^{\frac{-3\pi i}{2}\left\{-\frac{k(\alpha\tau+\beta)}{\gamma\tau+\delta}+k\alpha(\alpha\tau+\beta)+\frac{2\gamma a(\alpha\tau+\beta)}{c}+\frac{\gamma a^2(\gamma\tau+\delta)}{kc^2}\right\}}
	\no	\\&\quad\times	\widehat{A_3}\left(k(\alpha\tau+\beta)+\frac{a(\gamma\tau+\delta)}{c},0;2k\tau\right)\label{1261918},
	\end{align}
where the last equality follows from 
\begin{align}
	\gamma(\alpha\tau+\beta)=\gamma \alpha \tau+\alpha\delta -1=\alpha(\gamma\tau+\delta)-1.\label{alde}
	\end{align}

Noting that $\gamma\equiv 0\pmod{2kc},\alpha\equiv 1\pmod{2}, \delta\equiv1\pmod{c}$ and using \eqref{appel}, we find that
\begin{align}
&\widehat{A_3}\left(k(\alpha\tau+\beta)+\frac{a(\gamma\tau+\delta)}{c},0;2k\tau\right)\no
	\\&=	\widehat{A_3}\left(k\tau+\frac{a}{c}+\left(\frac{\alpha-1}{2}+\frac{a\gamma}{2kc}\right)(2k\tau)+k\beta+\frac{a(\delta-1)}{c},0;2k\tau\right)\no\\&=
	(-1)^{\frac{\alpha-1}{2}+\frac{a\gamma}{2kc}+k\beta+\frac{a(\delta-1)}{c}}e^{6\pi i\left(k\tau+\frac{a}{c}\right)\left(\frac{\alpha-1}{2}+\frac{a\gamma}{2kc}\right)+6k\pi i\tau\left(\frac{\alpha-1}{2}+\frac{a\gamma}{2kc}\right)^2}
	\no	\\&\quad\times	\widehat{A_3}\left(k\tau+\frac{a}{c},0;2k\tau\right).\label{1261917}
\end{align}
Substituting \eqref{1261917} into \eqref{1261918} and simplifying yields
\begin{align}
	&		F_k\left(a,c;B\tau\right)\no\\
	&=(\gamma\tau+\delta)e^{\frac{3k\pi i}{2}\frac{\alpha\tau+\beta}{\gamma\tau+\delta}-\frac{3k\pi i\tau}{2}}		(-1)^{\frac{\alpha-1}{2}(1+k\beta)+\frac{a\gamma}{2kc}+k\beta+\frac{a^2\gamma(\delta-1)}{2kc^2}}
		\no	\\&\quad\times e^{\frac{-3\pi i}{2}\left[k\beta-\frac{\gamma a^2}{kc^2}\right]}		F_k\left(a,c;\tau\right).\no
\end{align}
This together with \eqref{treta} and \eqref{ptrr} gives \eqref{kr}.

\end{proof}

To study the cusp conditions of $\widehat{\mathscr{R}^0}(a,c ; kB\tau)$, we need its transformation formula under $\mathrm{SL}_{2}(\mathbb{Z})$. We suppose $\alpha$ and $\gamma$ are integers with $\gcd(\alpha,\gamma)=1$ and take $B=\left(\begin{smallmatrix}\alpha&\beta\\ \gamma&\delta\end{smallmatrix}\right)\in\mathrm{SL}_{2}(\mathbb{Z})$. Let $j(B,\tau):=\gamma \tau+\delta$. Then we can verify that
\begin{align}
	j(BC,\tau)=j(B, C\tau)j(C,\tau).\no
\end{align}
For positive integers $k,$ we set $g=\gcd(\gamma,6k)$, $m=\frac{6k\alpha}{g}$, and $n=\frac{\gamma}{g}$. Since $\gcd(m,n)=1$, we take $L=(\begin{smallmatrix}m&s\\ n&t\end{smallmatrix})\in\mathrm{SL}_{2}(\mathbb{Z})$. Set
\[
C=L^{-1}\left(\begin{matrix}6k\alpha&6k\beta\\ \gamma&\delta\end{matrix}\right)=\left(\begin{matrix}g&6k\beta t-\delta s\\ 0&6k/g\end{matrix}\right).
\]
\begin{proposition}\label{prorc}
	Let $\zeta^*$ be some root of unity (might represent different values in different places).  
	Then we have
	\begin{align}
		&(\gamma\tau+\delta)^{-\frac{1}{2}}	\widehat{\mathscr{R}^0}(a,c ; kB\tau)\no
		\\&=\zeta^*\sqrt{\frac{g}{6k}}
		\exp\left(\pi i C\tau\left(\frac{anm}{c}-\frac{m^2}{36}-\frac{9a^2n^2}{c^2}\right)\right)
		\no	\\&\quad\times\widehat{\mu}\left(\left(\frac{m}{2}-\frac{an}{c}\right)C\tau+\frac{s}{2}-\frac{at}{c},\left(\frac{m}{3}+\frac{2an}{c}\right)C\tau+\frac{s}{3}+\frac{2at}{c};C\tau\right)\no
		\\&\quad+\zeta^*\sqrt{\frac{g}{6k}}
		\exp\left(-\pi i C\tau\left(\frac{anm}{c}+\frac{m^2}{36}+\frac{9a^2n^2}{c^2}\right)\right)
		\no	\\&\quad\times\widehat{\mu}\left(\left(\frac{m}{6}-\frac{an}{c}\right)C\tau+\frac{s}{6}-\frac{at}{c},\left(\frac{m}{3}+\frac{2an}{c}\right)C\tau+\frac{s}{3}+\frac{2at}{c};C\tau\right)\label{rtrc}.
	\end{align} 
\end{proposition}
\begin{proof}
	We verify that $6kB\tau=LC\tau$ and $(\gamma\tau+\delta)=j(LC,\tau)=j(L,C\tau)\cdot j(C,\tau)=j(L,C\tau)\cdot\frac{6k}{g}$. 
	Thus 
	\begin{align}
		j(L,C\tau)=nC\tau+t
		\intertext{and}
		LC\tau	j(L,C\tau)=mC\tau+s
	\end{align}
	Applying \eqref{trmu} yields
	\begin{align*}
		&\widehat{\mu}\left(3kB\tau-\frac{a}{c},\frac{2a}{c}+2kB\tau;6kB\tau\right)\\&=\widehat{\mu}\left(\frac{LC\tau}{2}-\frac{a}{c},\frac{2a}{c}+\frac{LC\tau}{3};LC\tau\right)
		\\&=\widehat{\mu}\left(\left(\frac{LC\tau}{2}-\frac{a}{c}\right)\frac{j(L,C\tau)}{j(L,C\tau)},\left(\frac{2a}{c}+\frac{LC\tau}{3}\right)\frac{j(L,C\tau)}{j(L,C\tau)};LC\tau\right)
		\\&=\widehat{\mu}\left(\frac{mC\tau+s}{2j(L,C\tau)}-\frac{a(nC\tau+t)}{cj(L,C\tau)},\frac{2a(nC\tau+t)}{cj(L,C\tau)}+\frac{mC\tau+s}{3j(L,C\tau)};\frac{mC\tau+s}{nC\tau+t}\right)
		\\&=\sqrt{\frac{(\gamma\tau+\delta)g}{6k}}\nu(L)^{-3}
		\exp\left(-\frac{n\pi i\left[\frac{mC\tau+s}{6}-\frac{3a(nC\tau+t)}{c}\right]^2}{nC\tau+t}\right)
		\\&\quad\times\widehat{\mu}\left(\left(\frac{m}{2}-\frac{an}{c}\right)C\tau+\frac{s}{2}-\frac{at}{c},\left(\frac{m}{3}+\frac{2an}{c}\right)C\tau+\frac{s}{3}+\frac{2at}{c};C\tau\right)
		\\&=\zeta^*\sqrt{\frac{(\gamma\tau+\delta)g}{6k}}
		\exp\left(\frac{\pi i(mC\tau+s)}{36(nC\tau+t)}+\pi i C\tau\left(\frac{anm}{c}-\frac{m^2}{36}-\frac{9a^2n^2}{c^2}\right)\right)
		\\&\quad\times\widehat{\mu}\left(\left(\frac{m}{2}-\frac{an}{c}\right)C\tau+\frac{s}{2}-\frac{at}{c},\left(\frac{m}{3}+\frac{2an}{c}\right)C\tau+\frac{s}{3}+\frac{2at}{c};C\tau\right)\label{u1}
	\end{align*}
	It follows 
	\begin{align}
		&(\gamma\tau+\delta)^{-\frac{1}{2}}e^{-\frac{\pi ikB\tau}{6}}\widehat{\mu}\left(3kB\tau-\frac{a}{c},\frac{2a}{c}+2kB\tau;6kB\tau\right)\no\\&=(\gamma\tau+\delta)^{-\frac{1}{2}}e^{-\frac{\pi iLC\tau}{36}}\widehat{\mu}\left(3kB\tau-\frac{a}{c},\frac{2a}{c}+2kB\tau;6kB\tau\right)
		\no	\\&=\zeta^*\sqrt{\frac{g}{6k}}
		\exp\left(\pi i C\tau\left(\frac{anm}{c}-\frac{m^2}{36}-\frac{9a^2n^2}{c^2}\right)\right)
		\no	\\&\quad\times\widehat{\mu}\left(\left(\frac{m}{2}-\frac{an}{c}\right)C\tau+\frac{s}{2}-\frac{at}{c},\left(\frac{m}{3}+\frac{2an}{c}\right)C\tau+\frac{s}{3}+\frac{2at}{c};C\tau\right)
	\end{align}
	Similarly, there holds
	\begin{align}
		&(\gamma\tau+\delta)^{-\frac{1}{2}}e^{-\frac{\pi ikB\tau}{6}}\widehat{\mu}\left(kB\tau-\frac{a}{c},\frac{2a}{c}+2kB\tau;6kB\tau\right)
		\no	\\&=\zeta^*\sqrt{\frac{g}{6k}}
		\exp\left(-\pi i C\tau\left(\frac{anm}{c}+\frac{m^2}{36}+\frac{9a^2n^2}{c^2}\right)\right)
		\no	\\&\quad\times\widehat{\mu}\left(\left(\frac{m}{6}-\frac{an}{c}\right)C\tau+\frac{s}{6}-\frac{at}{c},\left(\frac{m}{3}+\frac{2an}{c}\right)C\tau+\frac{s}{3}+\frac{2at}{c};C\tau\right).\label{u2}
	\end{align}
	By \eqref{ru1}, we have  
	\begin{align}
		&(\gamma\tau+\delta)^{-\frac{1}{2}}	\widehat{\mathscr{R}^0}(a,c ; kB\tau)
		\no\\	&=-i \zeta_{2c}^a (\gamma\tau+\delta)^{-\frac{1}{2}}e^{-\frac{\pi ikB\tau}{6}}\widehat{\mu}\left(3kB\tau-\frac{a}{c},\frac{2a}{c}+2kB\tau;6kB\tau\right)
		\no\\&\quad-i \zeta_{2c}^{-a}(\gamma\tau+\delta)^{-\frac{1}{2}}e^{-\frac{\pi ikB\tau}{6}}\widehat{\mu}\left(kB\tau-\frac{a}{c},\frac{2a}{c}+2kB\tau;6kB\tau\right).\label{prtrc}
	\end{align}
	Substituting \eqref{u1} and \eqref{u2} into \eqref{prtrc} gives \eqref{rtrc}.
\end{proof}

	Next, we study transformation formulas of $\widetilde{\mathcal{S}}(r, c ; \tau)$.
	To do this, we first prove
	\begin{proposition}
		Let 
		\begin{align}
			\widehat{	G}(r,c;\tau):=	q^{-\frac{(6r+2-c)(6r+2-5c)}{12}}
			\widehat{A_1}\left(c(6r+2-c)\tau-\frac{c-1}{2},2c^2\tau;6c^2\tau\right).\label{defda}
		\end{align}
		Then 
		\begin{align}
			\widehat{	G}(r,c;B\tau)
			&=(\gamma\tau+\delta)\frac{\nu^2\left(B\right)\nu\left({}^{2c^2}B\right)}{\nu^5\left({}^{2}B\right)}\widehat{	G}(r,c;\tau).\label{dago}
		\end{align}
		with
		$B=\left(\begin{array}{ll}
			\alpha & \beta \\
			\gamma &\delta
		\end{array}\right)\in\Gamma_0\left([\gcd(c,2)]^212c^2\right)\cap\Gamma_1(6c).$
	\end{proposition}
	\begin{proof}
		By \eqref{appmo}, for
		$B=\left(\begin{array}{ll}
			\alpha & \beta \\
			\gamma &\delta
		\end{array}\right)\in\Gamma_0\left([\gcd(c,2)]^212c^2\right)\cap\Gamma_1(6c),$  we have
		\begin{align}
			&	e^{\frac{(6r+2-c)(6r+2-5c)\pi i}{6}\left(\frac{\alpha\tau+\beta}{\gamma\tau+\delta}\no
				\right)}	\widehat{	G}(r,c;B\tau)\\&	=
			\widehat{A_1}\left(c(6r+2-c)B\tau-\frac{c-1}{2},2c^2B\tau;6c^2B\tau\right)\no\\&=	\widehat{A_1}\left(c(6r+2-c)\frac{\alpha\tau+\beta}{\gamma\tau+\delta}-\frac{(c-1)(\gamma\tau+\delta)}{2(\gamma\tau+\delta)},2c^2\frac{\alpha\tau+\beta}{\gamma\tau+\delta};\frac{\alpha(6c^2\tau)+6c^2\beta}{\frac{\gamma}{6c^2}(6c^2\tau)+\delta}\right)\no\\
			&=(\gamma\tau+\delta)e^{\frac{\pi i\gamma}{6c^2(\gamma\tau+\delta)}\left\{-\left[c(6r+2-c)(\alpha\tau+\beta)-\frac{(c-1)(\gamma\tau+\delta)}{2}\right]^2+2\left[c(6r+2-c)(\alpha\tau+\beta)-\frac{(c-1)(\gamma\tau+\delta)}{2}\right]\left[2c^2(\alpha\tau+\beta)\right]\right\}}
			\no\\&\quad	\times	\widehat{A_1}\left(c(6r+2-c)(\alpha\tau+\beta)-\frac{(c-1)(\gamma\tau+\delta)}{2},2c^2(\alpha\tau+\beta);6c^2\tau\right) \no\\
			&=(\gamma\tau+\delta)q^{\frac{\alpha(6r+2-c)}{12}\left(\frac{(c-1)\gamma}{c}+\alpha (5c-6r-2)\right)-\frac{\gamma}{12c^2}\left[\frac{(c-1)^2\gamma}{4}+2c^2(c-1)\alpha\right]}\no\\
			&\quad\times	e^{\frac{(6r+2-c)\pi i}{6}\left\{(6r+2-5c)\frac{\alpha\tau+\beta}{\gamma\tau+\delta}
				+\frac{\beta(c-1)\gamma}{c}+\alpha\beta(5c-6r-2)\right\}}
			\no	\\&\quad\times 	\widehat{A_1}\left(c(6r+2-c)(\alpha\tau+\beta)-\frac{(c-1)(\gamma\tau+\delta)}{2},2c^2(\alpha\tau+\beta);6c^2\tau\right),\label{sa1}
		\end{align}
		where the last equality follows from 
		\eqref{alde}.
		
		Noting that $\gamma\equiv 0\pmod{12c^2},\alpha\equiv 1\pmod{6c}, \delta\equiv1\pmod{2}$ and using \eqref{appel}, we find that
		\begin{align}
			&\widehat{A_1}\left(c(6r+2-c)(\alpha\tau+\beta)-\frac{(c-1)(\gamma\tau+\delta)}{2},2c^2(\alpha\tau+\beta);6c^2\tau\right)\no
			\\&=	\widehat{A_1}\left(c(6r+2-c)\tau-\frac{c-1}{2}+k(6c^2\tau)+l,2c^2\tau+m(6c^2\tau)+n;6c^2\tau\right)\no\\&=
			(-1)^{k+l}e^{2\pi i\left[c(6r+2-c)\tau-\frac{c-1}{2}\right](k-m)-2k\pi i \left(2c^2\tau\right)}q^{3c^2k^2-6c^2km}
			\no	\\&\quad\times	\widehat{A_1}\left(c(6r+2-c)\tau-\frac{c-1}{2},2c^2\tau;6c^2\tau\right),\label{1261917a1}
		\end{align}
		with $k=\frac{\alpha-1}{6c}(6r+2-c)-\frac{(c-1)\gamma}{12c^2},l=c(6r+2-c)\beta-\frac{(c-1)(\delta-1)}{2},m=\frac{\alpha-1}{3},n=2c^2\beta$.
		Substituting \eqref{1261917a1} into \eqref{sa1} and simplifying yields
		\begin{align}
			&		e^{\frac{(6r+2-c)(6r+2-5c)\pi i}{6}\left(\frac{\alpha\tau+\beta}{\gamma\tau+\delta}
				\right)}	\widehat{	G}(r,c;B\tau)\no\\
			&=(\gamma\tau+\delta)(-1)^{c\frac{\alpha-1}{6}\left(1+\beta\right)}e^{\frac{\pi i\beta(c^2-1)}{6}-\frac{\pi i\beta}{2}}	e^{\frac{(6r+2-c)(6r+2-5c)\pi i}{6}\left(\frac{\alpha\tau+\beta}{\gamma\tau+\delta}
				-\tau\right)}
			\no\\&\quad	\times	\widehat{A_1}\left(c(6r+2-c)\tau-\frac{c-1}{2},2c^2\tau;6c^2\tau\right).\no
		\end{align}
		Thus
		\begin{align}
			\widehat{	G}(r,c;B\tau)
			&=	(\gamma\tau+\delta)(-1)^{c\frac{\alpha-1}{6}\left(1+\beta\right)}e^{\frac{\pi i\beta(c^2-1)}{6}-\frac{\pi i\beta}{2}}\widehat{	G}(r,c;\tau)
			\no
		\end{align}
		with
		$B=\left(\begin{array}{ll}
			\alpha & \beta \\
			\gamma &\delta
		\end{array}\right)\in\Gamma_0\left([\gcd(c,2)]^212c^2\right)\cap\Gamma_1(6c).$
		Apply 
		\eqref{treta} to obtain 
		\begin{align}
			\frac{\nu\left({}^{2c^2}B\right)}{\nu\left({}^2B\right)}=	e^{\frac{\pi i\beta(c^2-1)}{6}} \label{nv2202}
		\end{align}
		with
		$B=\left(\begin{array}{ll}
			\alpha & \beta \\
			\gamma &\delta
		\end{array}\right)\in\Gamma_0\left([\gcd(c,2)]^212c^2\right)\cap\Gamma_1(6c).$ It follows 
		\begin{align}
			\widehat{	G}(r,c;B\tau)
			&=	(\gamma\tau+\delta)(-1)^{c\frac{\alpha-1}{6}\left(1+\beta\right)}e^{-\frac{\pi i\beta}{2}}\frac{\nu\left({}^{2c^2}B\right)}{\nu\left({}^2B\right)}\widehat{	G}(r,c;\tau)
			\no.
		\end{align}
		This together with \eqref{etaj} gives 
		\eqref{dago}.
	\end{proof}
	\begin{corollary}
		For $B=\left(\begin{array}{ll}
			\alpha & \beta \\
			\gamma &\delta
		\end{array}\right)\in\Gamma_0\left([\gcd(c,2)]^212c^2\right)\cap\Gamma_1(6c),$ 	we have
		\begin{align}
			\widetilde{\mathcal{S}}(r, c ; B\tau)
			=\sqrt{\gamma\tau+\delta}	\frac{\nu^2\left(B\right)}{\nu^5\left({}^{2}B\right)}	\widetilde{\mathcal{S}}(r, c ;B\tau). \label{trs}
		\end{align}
	\end{corollary}
	\begin{proof}
		By \eqref{defda}, we have
		\begin{align*}
			\frac{\widehat{	G}(r,c;\tau)}{\eta(2c^2\tau)}&=	\frac{q^{-\frac{(6r+2-c)(6r+2-5c)}{12}}
			}{\eta(2c^2\tau)}	\widehat{A_1}\left(c(6r+2-c)\tau-\frac{c-1}{2},2c^2\tau;6c^2\tau\right)\no\\
			&=\frac{q^{-\frac{(6r+2-c)(6r+2-5c)}{12}}
			}{\eta(2c^2\tau)}A_1\left(c(6r+2-c)\tau-\frac{c-1}{2},2c^2\tau;6c^2\tau\right)\no\\
			&\quad+\frac{iq^{-\frac{(6r+2-c)(6r+2-5c)}{12}}
				\vartheta(2c^2\tau ; 6c^2 \tau)	}{2\eta(2c^2\tau)} R\left(c (6r+2-3c)\tau-\frac{c-1}{2} ;6c^2 \tau\right)\no\\&=
		q^{-\frac{1}{3}}	S(r, c ; \tau)+\frac{1}{2}q^{-\frac{(3c-2-6r)^2}{12}}  R\left(c (6r+2-3c)\tau-\frac{c-1}{2} ;6c^2 \tau\right)
			\\&=	\widetilde{\mathcal{S}}(r, c ; \tau),
		\end{align*}
		which together with \eqref{dago} gives \eqref{trs}.
	\end{proof}

	Next, we study transformation properties of $\widetilde{\mathcal{S}}(r, c ; \tau)$ under $\mathrm{SL}_{2}(\mathbb{Z})$. For integers $\alpha$ and $\gamma$ with $\gcd(\alpha,\gamma)=1$,  we set $\tilde{g}=\gcd(\gamma,6c^2)$, $\tilde{m}=\frac{6c^2\alpha}{\tilde{g}}$, and $\tilde{n}=\frac{\gamma}{\tilde{g}}$. Take $\tilde{L}=(\begin{smallmatrix}\tilde{m}&\tilde{s}\\ \tilde{n}&\tilde{t}\end{smallmatrix})\in\mathrm{SL}_{2}(\mathbb{Z})$. Set
	\[
	\tilde{C}=\tilde{L}^{-1}\left(\begin{matrix}6c^2\alpha&6c^2\beta\\ \gamma&\delta\end{matrix}\right)=\left(\begin{matrix}\tilde{g}&6c^2\beta \tilde{t}-\delta \tilde{s}\\ 0&6c^2/\tilde{g}\end{matrix}\right).
	\]
	\begin{proposition}\label{prosc}
	We have
	\begin{align}
		&(\gamma\tau+\delta)^{-\frac{1}{2}}\widetilde{\mathcal{S}}(r, c ; B\tau)
		\no	\\&=\zeta^*\sqrt{\frac{\tilde{g}}{6c^2}}
		\exp\left(\pi i C\tau\left(\frac{\tilde{n}\tilde{m}(6r+2-3c)(c-1)}{6c}-\frac{\tilde{m}^2(6r+2-3c)^2}{36c^2}-\frac{\tilde{n}^2(c-1)^2}{4}\right)\right)
		\no	\\&\quad\times\widehat{\mu}\left(\left(\frac{\tilde{m}(6r+2-c)}{6c}-\frac{(c-1)\tilde{n}}{2}\right)C\tau+\frac{\tilde{s}(6r+2-c)}{6c}-\frac{(c-1)\tilde{t}}{2},\frac{\tilde{m}}{3}C\tau+\frac{\tilde{s}}{3};C\tau\right)
		.\label{trcs}
	\end{align}	
	\end{proposition}
	\begin{proof}
	Note that
	\begin{align*}
		\widetilde{\mathcal{S}}(r, c ; \tau)&=q^{-\frac{1}{3}}S(r, c ; \tau)+\frac{1}{2}q^{-\frac{(3c-2-6r)^2}{12}}  R\left(c (6r+2-3c)\tau-\frac{c-1}{2} ;6c^2 \tau\right) 
		\\&=\frac{i^{1-c}q^{c(6r+2-c)-\frac{(1+3r)^2 }{3} }}{\left( q^{2 c^2} ; q^{2 c^2}\right)_{\infty}} \sum_{n=-\infty}^{\infty} \frac{(-1)^{n } q^{3 c^2 n^2+5 c^2 n}}{1+(-1)^cq^{6 c^2 n+(2+6r-c) c}}
		\\&\quad	+\frac{1}{2}q^{-\frac{(3c-2-6r)^2}{12}}  R\left(c (6r+2-3c)\tau-\frac{c-1}{2} ;6c^2 \tau\right)
		\\&=-i q^{-\frac{(3c-2-6r)^2}{12}}\widehat{\mu}\left(c(6r+2-c)\tau-\frac{c-1}{2},2c^2\tau;6c^2\tau\right).
	\end{align*}
	Proceeding as in the proof of \eqref{rtrc}, one can apply \eqref{trmu} to obtain \eqref{trcs}.
	\end{proof}
\subsection{Proof of Theorem \ref{thm1}}	
Applying Proposition \ref{pro1514} with $k=18$ and $B=\left(\begin{array}{ll}
	\alpha & \beta \\
	\gamma &\delta
\end{array}\right)\in\Gamma_0(864)\cap\Gamma_0(72c^2)\cap\Gamma_1(8)\cap\Gamma_1(4c)$ and simplifying gives
	\begin{align}
	\widehat{\mathscr{R}^0}(a,c ; 18B\tau)&=\frac{\sqrt{\gamma\tau+\delta}}{\nu\left({}^{36}B\right)}	(-1)^{\beta}
		\widehat{\mathscr{R}^0}(a,c ; 18\tau).\no
\end{align}
A straightforward calculation gives 
	\begin{align}
\nu\left({}^{36}B\right)=\left(\frac{\gamma}{\delta}\right) 	(-1)^{\beta}\no.
\end{align}
Thus $	\widehat{\mathscr{R}^0}(a,c ; 18\tau)$ transforms correctly under $\Gamma_0(72c^2)\cap\Gamma_1(8)\cap\Gamma_1(4c)$.

Next, we show that $\widehat{\mathscr{R}^0}(a,c ; 18\tau)$ is annihilated by $\Delta_{\frac{1}{2}}$. By \eqref{cnu} and \eqref{ru1}, we have
	\begin{align}
	&	\widehat{\mathscr{R}^0}(a,c ;18 \tau)\no\\
	&=-i q^{-\frac{3}{2}}\zeta_{2c}^a\widehat{	\mu}\left(54\tau-\frac{a}{c}, \frac{2a}{c}+36\tau ; 108\tau\right)
\no\\&\quad	-i q^{-\frac{3}{2}}\zeta_{2c}^{-a}\widehat{	\mu}\left(18\tau-\frac{a}{c}, \frac{2a}{c}+36\tau ; 108\tau\right)\no\\
	&=-i q^{-\frac{3}{2}}\zeta_{2c}^a\mu\left(54\tau-\frac{a}{c}, \frac{2a}{c}+36\tau ; 108\tau\right)
\no\\&\quad	-i q^{-\frac{3}{2}}\zeta_{2c}^{-a}	\mu\left(18\tau-\frac{a}{c}, \frac{2a}{c}+36\tau ; 108\tau\right)\no\\
&\quad+\frac{q^{-\frac{3}{2}}\zeta_{2c}^a}{2}R\left(18\tau-\frac{3a}{c};108\tau\right)+\frac{q^{-\frac{3}{2}}\zeta_{2c}^{-a}}{2}R\left(-18\tau-\frac{3a}{c};108\tau\right).
\end{align}
Noting that
\[
\frac{\partial}{\partial \tau} = \frac{1}{2} \left( \frac{\partial}{\partial x} - i \frac{\partial}{\partial y} \right)
\quad \text{and} \quad
\frac{\partial}{\partial \overline{\tau}} = \frac{1}{2} \left( \frac{\partial}{\partial x} + i \frac{\partial}{\partial y} \right),
\]
we have
\[
\Delta_{\frac{1}{2}}= -4y^{\frac{3}{2}} \frac{\partial}{\partial \tau} y^{\frac{1}{2}} \frac{\partial}{\partial \overline{\tau}}.
\]
Since $-i q^{-\frac{3}{2}}\zeta_{2c}^a\mu\left(54\tau-\frac{a}{c}, \frac{2a}{c}+36\tau ; 108\tau\right)
	-i q^{-\frac{3}{2}}\zeta_{2c}^{-a}	\mu\left(18\tau-\frac{a}{c}, \frac{2a}{c}+36\tau ; 108\tau\right)$ is holomorphic in $\tau$, 
it suffices to show that $$\frac{q^{-\frac{3}{2}}\zeta_{2c}^a}{2}R\left(18\tau-\frac{3a}{c};108\tau\right)+\frac{q^{-\frac{3}{2}}\zeta_{2c}^{-a}}{2}R\left(-18\tau-\frac{3a}{c};108\tau\right)$$
is annihilated by $\Delta_{\frac{1}{2}}$. Applying \cite[Lemma 1.8]{mock}, after some direct calculations, we find that
\begin{align*}
&y^{\frac{1}{2}} \frac{\partial}{\partial \overline{\tau}	}
\left(\frac{q^{-\frac{3}{2}}\zeta_{2c}^a}{2}R\left(18\tau-\frac{3a}{c};108\tau\right)+\frac{q^{-\frac{3}{2}}\zeta_{2c}^{-a}}{2}R\left(-18\tau-\frac{3a}{c};108\tau\right)\right)
\no\\
&=\frac{-i\sqrt{54}\zeta_{2c}^a}{2} e^{-3\pi i \overline{\tau}}
\sum_{\nu\in\frac{1}{2}+\mathbb{Z}} (-1)^{\nu-\frac{1}{2}}\left(\nu+\frac{1}{6}\right)e^{-108\pi i \nu^2 \overline{\tau}-2\pi i\nu\left(18\overline{\tau}-\frac{3a}{c}\right)}
\no\\&\quad
-\frac{i\sqrt{54}\zeta_{2c}^{-a}}{2} e^{-3\pi i \overline{\tau}}
\sum_{\nu\in\frac{1}{2}+\mathbb{Z}} (-1)^{\nu-\frac{1}{2}}\left(\nu-\frac{1}{6}\right)e^{-108\pi i \nu^2 \overline{\tau}+2\pi i\nu\left(18\overline{\tau}+\frac{3a}{c}\right)},
\end{align*}
which is holomorphic in $\overline{\tau}$. Thus it is annihilated by $\frac{\partial}{\partial \tau}$ and the results follows.

Lastly, we check the growth conditions of $\widehat{\mathscr{R}^0}(a,c ; 18\tau)$ at the cusps. Suppose $\alpha$ and $\gamma$ are integers with $\gcd(\alpha,\gamma)=1$ and take $B=\left(\begin{smallmatrix}\alpha&\beta\\ \gamma&\delta\end{smallmatrix}\right)\in\mathrm{SL}_{2}(\mathbb{Z})$. Set $g=\gcd(\gamma,108)$, $m=\frac{108\alpha}{g}$, and $n=\frac{\gamma}{g}$. Since $\gcd(m,n)=1$, we take $L=(\begin{smallmatrix}m&s\\ n&t\end{smallmatrix})\in\mathrm{SL}_{2}(\mathbb{Z})$. Set
\[
C=L^{-1}\left(\begin{matrix}108\alpha&108\beta\\ \gamma&\delta\end{matrix}\right)=\left(\begin{matrix}g&108\beta t-\delta s\\ 0&108/g\end{matrix}\right).
\]
Then Proposition \ref{prorc} gives
	\begin{align}
	&(\gamma\tau+\delta)^{-\frac{1}{2}}	\widehat{\mathscr{R}^0}(a,c ; 18B\tau)\no
	\\&=\zeta^*\sqrt{\frac{g}{108}}
	\exp\left(\pi i C\tau\left(\frac{anm}{c}-\frac{m^2}{36}-\frac{9a^2n^2}{c^2}\right)\right)
	\no	\\&\quad\times\widehat{\mu}\left(\left(\frac{m}{2}-\frac{an}{c}\right)C\tau+\frac{s}{2}-\frac{at}{c},\left(\frac{m}{3}+\frac{2an}{c}\right)C\tau+\frac{s}{3}+\frac{2at}{c};C\tau\right)\no
	\\&\quad+\zeta^*\sqrt{\frac{g}{108}}
	\exp\left(-\pi i C\tau\left(\frac{anm}{c}+\frac{m^2}{36}+\frac{9a^2n^2}{c^2}\right)\right)
	\no	\\&\quad\times\widehat{\mu}\left(\left(\frac{m}{6}-\frac{an}{c}\right)C\tau+\frac{s}{6}-\frac{at}{c},\left(\frac{m}{3}+\frac{2an}{c}\right)C\tau+\frac{s}{3}+\frac{2at}{c};C\tau\right),\no
\end{align} 
where $\zeta^*$ is some root of unity. Noting that $\exp\left(\pi i C\tau\right)=\zeta^* q^{\frac{g^2}{216}}$, after applying \eqref{gth} and \eqref{cnu}, we find that
$(\gamma\tau+\delta)^{-\frac{1}{2}}	\widehat{\mathscr{R}^0}(a,c ; 18B\tau)$ has at worst linear exponential growth as $y\rightarrow \infty.$
This completes the proof of Theorem \ref{thm1}.

Proceeding as above, we can apply Propositions \ref{pro1514}, \ref{prorc} and \eqref{etaj} to obtain
\begin{theorem}
	For positive integers $0<a<c$, $	\widehat{\mathscr{R}^0}(a,c ; \tau)$ is a weak Maass form of weight $1/2$ on $\Gamma_0(2c)\cap\Gamma_1(c)$ with mutiplier $$\frac{\nu^2\left(B\right)}{\nu^5\left({}^{2}B\right)}	(-1)^{\frac{a\gamma}{2c}+\frac{a^2\gamma(\delta-1)}{2c^2}}
	e^{\frac{3\gamma a^2\pi i}{2c^2}}		$$
	for $B=\left(\begin{array}{ll}
		\alpha & \beta \\
		\gamma &\delta
	\end{array}\right).$ In particular, if $B=\left(\begin{array}{ll}
		\alpha & \beta \\
		\gamma &\delta
	\end{array}\right)\in\Gamma_0(4c^2)\cap\Gamma_1(c)$, then
	\begin{align}
		\widehat{\mathscr{R}^0}(a,c ; B\tau)=\frac{\nu^2\left(B\right)}{\nu^5\left({}^{2}B\right)}	\sqrt{\gamma\tau+\delta}\widehat{\mathscr{R}^0}(a,c ; \tau)\label{8291116}
	\end{align}
\end{theorem}
%

\subsection{Proof of Theorem \ref{thm2}}
We consider part (i) of Theorem \ref{thm2}. By \eqref{holo}, we have
\begin{align}
&	\frac{\eta^5\left(2\tau\right)}{\eta^2(\tau)}\left(	\mathscr{R}^0(a,c; \tau)-i^{1-c}q^{-\frac{1}{3}}\sum_{r=0}^{c-1}(-1)^r\left(\zeta_c^{(3r+1)a}+\zeta_c^{-(3r+1)a}\right)S(r, c ; \tau)\right)\no	\\&=	\frac{\eta^5\left(2\tau\right)}{\eta^2(\tau)}\left(	\widehat{\mathscr{R}^0}(a,c; \tau)-i^{1-c}\sum_{r=0}^{c-1}(-1)^r\left(\zeta_c^{(3r+1)a}+\zeta_c^{-(3r+1)a}\right) \widetilde{\mathcal{S}}(r, c ; \tau)\right)\label{8291118}
\end{align}
	For $B\in\Gamma_0\left(2\right),$ 	we have
\begin{align}
\frac{\eta^5\left(2B\tau\right)}{\eta^2(B\tau)}
	=(\gamma\tau+\delta)^{\frac{3}{2}}\frac{\nu^5\left({}^{2}B\right)}{\nu^2\left(B\right)}\frac{\eta^5\left(2\tau\right)}{\eta^2(\tau)}. \no
\end{align}
This together with \eqref{trs}, \eqref{8291116} and \eqref{8291118} implies that, the holomorphic function
\begin{align}
	\frac{\eta^5\left(2\tau\right)}{\eta^2(\tau)}\left(	\widehat{\mathscr{R}^0}(a,c; \tau)-i^{1-c}\sum_{r=0}^{c-1}(-1)^r\left(\zeta_c^{(3r+1)a}+\zeta_c^{-(3r+1)a}\right) \widetilde{\mathcal{S}}(r, c ; \tau)\right)\no
\end{align}
transforms correctly under $\Gamma_0\left([\gcd(c,2)]^212c^2\right)\cap\Gamma_1(6c).$ Then the result follows.

Similarly, we apply \eqref{holo} to find that
\begin{align}
M(\tau):&=			\frac{\eta^5\left(2c^2\tau\right)}{\eta^2(c^2\tau)}\left(	\mathscr{R}^0(a,c; \tau)-i^{1-c}q^{-\frac{1}{3}}\sum_{r=0}^{c-1}(-1)^r\left(\zeta_c^{(3r+1)a}+\zeta_c^{-(3r+1)a}\right)S(r, c ; \tau)\right)\no	\\&=		\frac{\eta^5\left(2c^2\tau\right)}{\eta^2(c^2\tau)}\left(	\widehat{\mathscr{R}^0}(a,c; \tau)-i^{1-c}\sum_{r=0}^{c-1}(-1)^r\left(\zeta_c^{(3r+1)a}+\zeta_c^{-(3r+1)a}\right) \widetilde{\mathcal{S}}(r, c ; \tau)\right)\no.
\end{align}
Invoke \eqref{treta}, \eqref{trs} and \eqref{8291116} to obtain
\begin{align}
(\gamma\tau+\delta)^{-2}	M(B\tau)&=\frac{\nu^5\left({}^{2c^2}B\right)\nu^2\left(B\right)}{\nu^2\left({}^{c^2}B\right)\nu^5\left({}^{2}B\right)}M(\tau)	\label{trm}		
\end{align}
for $B\in\Gamma_0\left([\gcd(c,2)]^212c^2\right)\cap\Gamma_1(6c)$. Applying \eqref{treta} yields 
	\begin{align}
	\frac{\nu^2\left(B\right)}{\nu^2\left({}^{c^2}B\right)}=	e^{\frac{\pi i\beta(1-c^2)}{6}} \no
\end{align}
which together with \eqref{nv2202} and \eqref{trm} gives
\begin{align}
	(\gamma\tau+\delta)^{-2}	M(B\tau)&=	e^{\frac{2\pi i\beta(c^2-1)}{3}} M(\tau)=M(\tau)\no	
\end{align}
since $3\nmid c$. Then the result follows. This completes the proof of Theorem \ref{thm2}.

\section{Proof of Theorem \ref{thm3} and Corollary \ref{coro}}\label{s4}
We prove Theorem \ref{thm3} by applying the valence formula.

To study the cusp conditions, we need \cite[Corollary 6.2]{cjs1}.
For a real number $w$, we let $\lfloor w \rfloor$ denote the greatest integer less than or equal to $w$ and $\{w\}$ the fractional part of $w$. That is, $w = \lfloor w \rfloor + \{w\}$, $\lfloor w \rfloor \in \mathbb{Z}$, and $0 \leq \{w\} < 1$.
\begin{proposition}\label{procuu}
	Let $u_1, u_2, w_1, w_2$ be real numbers. Then the lowest power of $q$ appearing in the expansion of the holomorphic part of $q^\alpha \widehat{\mu}\left(u_1\tau+u_2,w_1\tau+w_2;\tau\right)$ is at least $\alpha+\tilde{\nu}(u_1,w_1)$, where
	\begin{align*}
		\tilde{\nu}(u,w) &= \frac{1}{2} \left( \lfloor u \rfloor - \lfloor w \rfloor \right)^2 + \left( \lfloor u \rfloor - \lfloor w \rfloor \right) \left( \{ u \} - \{ w \} \right) + k(u,w),
		\\
		k(u,w)& = 
		\begin{cases}
			\nu(\{ u \}, \{ w \}) & \text{if } \{ u \} - \{ w \} \neq \pm \frac{1}{2}, \\
			\min \left( \frac{1}{8}, \nu(\{ u \}, \{ w \}) \right) & \text{if } \{ u \} - \{ w \} = \pm \frac{1}{2},
		\end{cases}
		\\
		\nu(u,w) &=
		\begin{cases}
			\frac{u + w}{2} - \frac{1}{8} & \text{if } u + w \leq 1, \\
			\frac{7}{8} - \frac{u + w}{2} & \text{if } u + w > 1.
		\end{cases}
	\end{align*}	
\end{proposition}
Applying Propositions \ref{prorc}, \ref{prosc} and \ref{procuu}, we obtain the following immediately.
\begin{proposition}\label{ordpro}
	Let $\alpha, \gamma$ be integers with $\gcd(\alpha,\gamma)=1$.
	\begin{enumerate}[(i)]
		\item 
		The lowest power of $q$ appearing in the expansion of the holomorphic part of 	$\widehat{\mathscr{R}^0}(a,c ; \tau)$ at the cusp $\frac{\alpha}{\gamma}$ is at least 
		\begin{align}
			\frac{g^2}{6}\Bigg(&
			\min\left\{\frac{anm}{2c}+\tilde{\nu}\left(\frac{m}{2}-\frac{an}{c},\frac{m}{3}+\frac{2an}{c}\right),-\frac{anm}{2c}+\tilde{\nu}\left(\frac{m}{6}-\frac{an}{c},\frac{m}{3}+\frac{2an}{c}\right)\right\}
			\no\\&\quad
			-\frac{m^2}{72}-\frac{9a^2n^2}{2c^2}\Bigg),
		\end{align}
		where $g=\gcd(6,\gamma), m=\frac{6\alpha}{g}, n=\frac{\gamma}{g}.$
		
		\item
		The lowest power of $q$ appearing in the expansion of the holomorphic part of 	$\widetilde{\mathcal{S}}(r, c ; \tau)$ at the cusp $\frac{\alpha}{\gamma}$ is at least 
		\begin{align}
			\frac{\tilde{g}^2}{6c^2}\Bigg(&
			\frac{\tilde{n}\tilde{m}(6r+2-3c)(c-1)}{12c}-\frac{\tilde{m}^2(6r+2-3c)^2}{72c^2}-\frac{\tilde{n}^2(c-1)^2}{8}
			\no\\&\quad+\tilde{\nu}\left(\frac{\tilde{m}(6r+2-c)}{6c}-\frac{(c-1)\tilde{n}}{2},\frac{\tilde{m}}{3}\right)
			\Bigg),
		\end{align}	
		where $\tilde{g}=\gcd(\gamma,6c^2)$, $\tilde{m}=\frac{6c^2\alpha}{\tilde{g}}$, and $\tilde{n}=\frac{\gamma}{\tilde{g}}$.
	\end{enumerate}
\end{proposition}

Now we are in a position to prove Theorem \ref{thm3}.
\begin{proof}[Proof of Theorem \ref{thm3}] 
Multiplying by $\frac{\eta^5(50\tau)\eta(30\tau)}{q^{\frac{1}{3}}\eta^2(25\tau)\eta^5(150\tau)}$ throughout, after rearranging and simplifying, we find that \eqref{r5} is equivalent to
	\begin{align}
&\no\frac{\eta(30\tau)}{\eta^5(150\tau)}	\times\left\{\frac{\eta^5\left(50\tau\right)}{\eta^2(25\tau)}\left(	\mathscr{R}^0(1,5; \tau)-q^{-\frac{1}{3}}\sum_{r=0}^{4}(-1)^r\left(\zeta_5^{3r+1}+\zeta_5^{-3r-1}\right)S(r, 5 ; \tau)\right)\right\}
\\ &-\left(\zeta_5^2+\zeta_5^3\right)\Bigg(
\frac{J_{30,150}J^2_{50,150}J_{60,150}}{q^{25}J_{15,150}J^2_{25,150}J_{75,150}}
-\frac{J^2_{15,150}J^2_{35,150}J^4_{50,150}J^2_{65,150}}{q^{30}J_{5,150}J_{10,150}J_{20,150}J^2_{25,150}J_{40,150}J_{45,150}J_{55,150}J_{70,150}J_{75,150}}\no\\ &\quad\qquad\quad\qquad\quad
-\frac{J_{5,150}J_{45,150}J^4_{50,150}J_{55,150}}{q^{24}J_{10,150}J_{20,150}J^2_{25,150}J_{40,150}J_{70,150}J_{75,150}}
-\frac{J^4_{50,150}}{q^{28}J_{10,150}J_{20,150}J_{40,150}J_{70,150}}
\no\\ &\quad\qquad\quad\qquad\quad
+\frac{J_{15,150}J_{35,150}J^4_{50,150}J_{65,150}}{q^{27}J_{10,150}J_{20,150}J^2_{25,150}J_{40,150}J_{70,150}J_{75,150}}\no\\ &\quad\qquad\quad\qquad\quad
+\frac{J^2_{5,150}J^2_{45,150}J^4_{50,150}J^2_{55,150}}
{q^{21}J_{10,150}J_{15,150}J_{20,150}J^2_{25,150}J_{35,150}J_{40,150}J_{65,150}J_{70,150}J_{75,150}}
\no\\ &\quad\qquad\quad\qquad\quad+\frac{J_{30,150}J^2_{50,150}J_{60,150}}{q^{16}J^2_{25,150}J_{45,150}J_{75,150}}\Bigg)
 \no
-\frac{J_{30,150}J^2_{50,150}J_{60,150}}{q^{25}J_{15,150}J^2_{25,150}J_{75,150}}
+\frac{J_{30,150}J^4_{50,150}}{q^{30}J^2_{10,150}J^2_{40,150}J_{60,150}}\no\\ &
-\frac{J_{15,150}J_{30,150}J_{35,150}J^4_{50,150}J_{65,150}}{q^{29}J^2_{10,150}J^2_{25,150}J^2_{40,150}J_{60,150}J_{75,150}}
+\frac{J^2_{15,150}J_{30,150}J^2_{35,150}J^4_{50,150}J^2_{65,150}}{q^{28}J^2_{10,150}J^4_{25,150}J^2_{40,150}J_{60,150}J^2_{75,150}}
\no\\ &
+\frac{J_{30,150}J^2_{50,150}J_{60,150}}{q^{13}J^2_{25,150}J^2_{75,150}}
-\frac{J_{5,150}J_{30,150}J_{45,150}J^4_{50,150}J_{55,150}}{q^{27}J^2_{10,150}J_{15,150}J_{35,150}
	J^2_{40,150}J_{60,150}J_{65,150}}\no\\&
	+\frac{J_{5,150}J_{30,150}J_{45,150}J^4_{50,150}J_{55,150}}{q^{26}J^2_{10,150}J^2_{25,150}J^2_{40,150}
		J_{60,150}J_{75,150}}
=0\label{eqth}.
\end{align}
Noting that $\frac{\eta^5(150\tau)}{\eta(30\tau)}\in M_{2}\left(\Gamma_{0}(150)\right)$, we apply Theorem \ref{thm2} (ii) to find that the first term on the left side of \eqref{eqth} is a modular function on  
$\Gamma_{1}(300)$. An application of \cite[Theorem 3]{sinai} reveals that each infinite product on the left side of \eqref{eqth} is also a modular function on  
$\Gamma_{1}(300)$ (we check this automatically by using Frey and Garvan’s
MAPLE package THETAIDS \cite{maple}).

Next, we prove \eqref{eqth} by applying the valence formula.
We first recall the following result which follows immediately from the valence formula for modular functions: suppose that $f$ is a modular function (holomorphic on $\mathbb{H}$) on some congruence subgroup $\Gamma\subset SL_2(\mathbb{Z})$ and
$$\sum_{\zeta\in \mathcal{D}} \textrm{ord}(f,\zeta)\cdot \textrm{width}_\Gamma(\zeta)>0,$$
where $\mathcal{D}$ is a set of inequivalent cusps for $\Gamma$, $\textrm{ord}(f,\zeta)$ denotes the invariant order of $f$ at $\zeta$ and $\textrm{width}_\Gamma(\zeta)$ is the width of the cusp $\zeta$ for $\Gamma$.
Then $f$ is identically to zero. See \cite[Theorem 4.1.4]{rank} for more details and the notation used.

To check the cusp conditions of the infinite products in \eqref{eqth}, we require the following results:
\begin{enumerate}[(i)]
	\item 
	\cite[Corollary 2.2]{eta}: 
	Let $N \geq 1$ and let
	$$
	f(z)=\prod_{m \mid N} \eta(m z)^{a_{m}},
	$$
	with $a_{m} \in \mathbb{Z}$. Then for $(a, c)=1$,
\begin{align}
	\operatorname{ord}\left(f(z) ; \frac{a}{c}\right)=\sum_{m \mid N} \frac{(m, c)^{2} a_{m}}{24 m}.
	\label{ordeta}
\end{align}
	
	\item
	\cite[Lemma 3.2]{ajf}: Let 
	$$\theta_{\delta ; g}(\tau):=q^{\frac{(\delta-2 g)^{2}} {8 \delta}} (q^g,q^{\delta-g},q^\delta;q^\delta)_\infty,$$ with $0<g<\delta$. Then for  $(b, c)=1$,
\begin{align}
	\operatorname{ord}\left(\theta_{g ; \delta}(\tau), s\right)=\frac{e^{2}}{2 \delta}\left(\frac{b g}{e}-\left[\frac{b g}{e}\right]-\frac{1}{2}\right)^{2},	\label{ordtheta}
\end{align}
	where $e=(\delta, c)$ and $[\quad ]$ is the greatest integer function.
\end{enumerate}
Applying Proposition \ref{ordpro}, \eqref{ordeta} and \eqref{ordtheta}, we obtain lower bounds on the orders of the left hand side of \eqref{eqth} at the non-infinite cusps of $\Gamma_{1}(300)$ by taking the minimum order of each of the individual summands. This yields
\begin{align}&\sum_{\zeta\in \mathcal{D}} \textrm{ord}(\textrm{L.H.S of \eqref{eqth}},\zeta)\cdot \textrm{width}_{\Gamma_{1}(300)}(\zeta)\no\\&\geq \textrm{ord}(\textrm{L.H.S of \eqref{eqth}},\infty)-11505.\no
	\end{align}
	However, we computationally verify the validity of equations \eqref{co01} and \eqref{co12} for the first $12,000$ coefficients in their $q$-expansions.
	Moreover, following the proof of Corollary \ref{coro}, one can deduce \eqref{r5} from \eqref{co01} and \eqref{co12}. Thus $$\textrm{ord}(\textrm{L.H.S of \eqref{eqth}},\infty)\geq 12000.$$
Then \begin{align}&\sum_{\zeta\in \mathcal{D}} \textrm{ord}(\textrm{L.H.S of \eqref{eqth}},\zeta)\cdot \textrm{width}_{\Gamma_{1}(300)}(\zeta)>0\no.
	\end{align}
	Thus L.H.S of \eqref{eqth} is identically to zero and \eqref{r5} follows.
\end{proof}	
Next, we apply Theorem \ref{thm3} to obtain Corollary \ref{coro}.
\begin{proof}[Proof of Corollary \ref{coro}]
	By definition, we have
	\begin{align}
		R^0(\zeta_5 ; q)=\sum_{n=1}^\infty\sum_{m=-\infty}^\infty N^0(m,n)\zeta_5^mq^n=\sum_{m=0}^4\sum_{n=1}^\infty N^0(m,5,n)\zeta_5^mq^n\no.
	\end{align}
Since $N^0(m,5,n)=N^0(5-m,5,n)$ and $\sum_{m=0}^4\zeta_5^m=0$, we have
	\begin{align}
	R^0(\zeta_5 ; q)&=\sum_{n=1}^\infty \left(N^0(0,5,n)-N^0(1,5,n)\right)q^n\no\\&\quad+\left(\zeta_5^2+\zeta_5^3\right)\sum_{n=1}^\infty \left(N^0(2,5,n)-N^0(1,5,n)\right)q^n.\label{pco}
\end{align}
Noting that the minimal polynomial 
of $\zeta_5$ over the rational numbers is degree of $4$ and replacing $\zeta_5+\zeta^4_5$ by $-1-\zeta^2_5-\zeta^3_5$ in \eqref{r5}, after comparing the coefficients of $\zeta_5^k$ in the resulting equation to those in \eqref{pco}, we obtain \eqref{co01} and \eqref{co12}.
\end{proof}

\section*{Acknowledgements}
This work is supported by NSFC (National Natural Science
Foundation of China) through Grants NSFC 12471315.

\end{document}